\documentclass[a4paper,11pt,reqno]{amsart}

\usepackage{a4wide}

\usepackage[utf8]{inputenc}
\usepackage{amsmath}
\usepackage{amssymb}
\usepackage{amsthm}
\usepackage{fancyhdr} 
\usepackage{graphicx} 
\usepackage{enumerate}
\usepackage{cite}

\allowdisplaybreaks
\usepackage[toc,page]{appendix}
\usepackage{hyperref}    
\usepackage{tikz-cd}

\usepackage[all]{xy}
\usepackage{latexsym}
\usepackage{amscd}

\newcommand{\Oo}{\mathcal{O}}
\newcommand{\ZZ}{\mathbb{Z}}

\newcommand{\RR}{\mathbb{R}}
\newcommand{\R}{\mathbb{R}}
\newcommand{\QQ}{\mathbb{Q}}
\newcommand{\CC}{\mathbb{C}}
\newcommand{\TT}{\mathbb{T}}
\newcommand{\Z}{\mathbb{Z}}
\newcommand{\NN}{\mathbb{N}}

\newcommand{\vol}{\textup{vol}}
\newcommand{\inte}{\textup{int}\,}
\newcommand{\Ehr}{\textup{Ehr}}
\newcommand{\G}{\mathcal G}
\newcommand{\T}{\mathcal T}
\renewcommand{\S}{\mathcal S}

\DeclareMathOperator{\codim}{codim}
\DeclareMathOperator{\bx}{Box}

\newcommand{\om}{\omega}

\newtheorem{theorem}{Theorem}[section]
\newtheorem{lemma}[theorem]{Lemma}
\newtheorem{thm}[theorem]{Theorem}
\newtheorem{proposition}[theorem]{Proposition}
\newtheorem{prop}[theorem]{Proposition}
\newtheorem{corollary}[theorem]{Corollary}
\newtheorem{cor}[theorem]{Corollary}
\newtheorem{claim}[theorem]{Claim}

\newtheorem{definition}[theorem]{Definition}
\newtheorem{defn}[theorem]{Definition}
\newtheorem{remark}[theorem]{Remark}
\newtheorem{rem}[theorem]{Remark}
\newtheorem{example}[theorem]{Example}

\DeclareMathOperator{\Chi}{{\displaystyle{\chi}}}

\DeclareMathOperator{\Aff}{Aff}
\DeclareMathOperator{\conv}{conv}
\newcommand{\cat}{\mathcal}
\newcommand{\Xx}{{\cat X}}
\newcommand{\teta}{{\tilde{\eta}}}

\begin{document}
\title[Contact  invariants  and  the  Ehrhart  polynomial]{Contact invariants of $\QQ$-Gorenstein
toric contact manifolds, the Ehrhart polynomial and Chen-Ruan cohomology}

\author[M.~Abreu]{Miguel Abreu}
\address{Center for Mathematical Analysis, Geometry and Dynamical Systems,
Instituto Superior T\'ecnico, Universidade de Lisboa, 
Av. Rovisco Pais, 1049-001 Lisboa, Portugal}
\email{mabreu@math.tecnico.ulisboa.pt, macarini@math.tecnico.ulisboa.pt}

\author[L.~Macarini]{Leonardo Macarini}
 
\author[M.~Moreira]{Miguel Moreira}
\address {ETH Z\"urich, Department of Mathematics. Rämistrasse 101, 8092 Zürich, Switzerland.}
\email{miguel.moreira@math.ethz.ch}
\thanks{Partially funded by FCT/Portugal through UID/MAT/04459/2020. MA and LM were 
also funded through project PTDC/MAT-PUR/29447/2017. MM received funding from the European Research Council (ERC)
under the European Unions Horizon 2020 research and innovation programme
(ERC-2017-AdG-786580-MACI)}

\begin{abstract}
$\QQ$-Gorenstein toric contact manifolds provide an interesting class of examples of contact 
manifolds with torsion first Chern class. They are completely determined by certain rational convex 
polytopes, called toric diagrams, and arise both as links of toric isolated singularities and as 
prequantizations of monotone toric symplectic orbifolds. In this paper we show how the cylindrical 
contact homology invariants of a $\QQ$-Gorenstein toric contact manifold are related to
\begin{itemize}
\item the Ehrhart (quasi-)polynomial of its toric diagram;
\item the Chen-Ruan cohomology of any crepant toric orbifold resolution 
of its corresponding toric isolated singularity;
\item the Chen-Ruan cohomology of any monotone toric symplectic orbifold
base that gives rise to it through prequantization.
\end{itemize}
\end{abstract}

\maketitle

\section{Introduction}

$\QQ$-Gorenstein toric contact manifolds are good toric contact manifolds with torsion first Chern class and
can be thought of as the odd dimensional analogues of monotone toric symplectic manifolds.
While the latter are in 1-1 correspondence with reflexive Delzant polytopes (up to translation and scaling), 
the former are in 1-1 correspondence with toric diagrams, i.e. rational simplicial polytopes with unimodular facets
(cf. Definition~\ref{def:diagram}).

Given a toric diagram $D\subset \RR^n$ and corresponding $\QQ$-Gorenstein toric contact manifold
$(M_D^{2n+1}, \xi_D)$, any point $v$ in the interior of $D$ determines a toric contact form
$\alpha_\nu$ and toric Reeb vector field $R_\nu$ (cf. Definition~\ref{def:reeb}).  
If $\nu=(v,1)$ has $\QQ$-independent
coordinates, then the Reeb flow of $R_\nu$ is non-degenerate with finitely many simple closed orbits
which are in 1-1 correspondence with the facets of $D$. The (contact homology) degree of any
closed $R_\nu$-orbit is equal to its Conley-Zehnder index plus $n-2$. Note that in this context of
torsion (not necessarily zero) first Chern class, the Conley-Zehnder index is a well-defined rational
(not necessarily integer) number.

\begin{definition} \label{def:Betti_Euler}
Let $D\subset \RR^n$ be a toric diagram of order $m\in\NN$, i.e. such that $mD$ is an integral simplicial
polytope with unimodular facets.
The \emph{contact Betti numbers} $cb_{j/m} (D, \nu)$, $j\in\ZZ$ are defined by
\[
cb_{j/m} (D, \nu) = \text{number of closed $R_\nu$-orbits with contact homology degree $j/m$}.
\]
\end{definition}

When $D$ is a toric diagram of order $1$, i.e. when $(M_D^{2n+1}, \xi_D)$ is a toric contact manifold
with zero first Chern class (Gorenstein), we proved in~\cite{abreu2012contact, abreu2016mean, abreu2018contact} 
the following relevant results for this paper:
\begin{itemize}
\item An explicit formula to compute the Conley-Zehnder index of each closed $R_\nu$-orbit using only
data contained in the pair $(D,\nu)$.
\item The degree of any closed $R_\nu$-orbit is always even and so $cb_j (D, \nu) = 0$ when $j$ is odd.
\item $\inf\{j\in\ZZ\,:\ cb_j (D,\nu)\ne 0\}$ is finite.
\item The mean Euler characteristic
\[
\Chi (D, \nu) := \lim_{N\to\infty} \frac{1}{N} \sum_{j=-N}^{N} (-1)^j cb_{j} (D,\nu) =
 \lim_{N\to\infty} \frac{1}{2N} \sum_{j=0}^{N} cb_{2j} (D,\nu) 
\]
is given by
\begin{equation} \label{eq:meanEC}
\Chi (D, \nu) = \frac{n! \vol (D)}{2}\,.
\end{equation}
\end{itemize}

Note that~(\ref{eq:meanEC}) shows that the mean Euler characteristic $\Chi (D, \nu)$ depends only on the 
Gorenstein toric contact manifold $(M_D^{2n+1}, \xi_D)$, via the volume of its toric diagram $D$, and not on
the particular choice of $v$ in the interior of $D$. Our first main result in this paper shows that, in fact, for 
a toric diagram $D$ of any order $m\in\NN$ each contact Betti number $cb_{j/m} (D, \nu)$ is also independent 
of $\nu$ and can be combinatorially determined using the Ehrhart quasi-polynomial of $D$.

Recall that the Ehrhart (quasi-)polynomial of $D\subset\RR^n$ counts its number of rational points and is defined by
\[
L_D (t)=\#\left(D\cap \frac{1}{t}\ZZ^n\right)\,,\ t\in\NN\,.
\]
It is well known that $L_D (t)$ is a quasi-polynomial of degree $n$ and period $m$, i.e.
\[
L_D (t) = \sum_{k=0}^n  c_k (D, t) t^k \,,
\]
where each $c_k (D, t+m) = c_k (D, t)$, for all $t\in\NN$ and $k=0,\ldots,n$. Moreover, $c_0 (D,0) =1$ and 
$c_n (D,t) = \vol (D)$ for all $t\in\NN$. A result of Stanley \cite[Theorems 1.6, 1.7]{stanley} says that its coefficients in the
quasi-polynomial basis $\{\binom{\frac{t-k}{m} +n}{n} : k=0,\ldots,m(n+1)-1\}$ are non-negative 
integers, i.e.
\[
L_D (t) = \sum_{k \in [0, m(n+1)-1]\,,\  k\equiv t \mod m}  \delta_k (D) \binom{\frac{t-k}{m}+n}{n} \,, \ \text{with}\ \delta_k (D) \in \NN_0\,.
\]
Defining $\delta_k =0$ for $k<0$ and $k>m(n+1)-1$, we can now state our first main result.
\begin{theorem} \label{thm:invariants}
\[
cb_{2k} (D,\nu)  - cb_{2(k-1)} (D,\nu) = \delta_{m(n-k)} (D)\,.
\]
\end{theorem}
\noindent Hence, the contact Betti numbers are indeed independent of $\nu$ and from now on we will simply denote them 
by $cb_j (D)$. 

Note that a toric Reeb vector $\nu$ corresponds to a point $v$ in the interior of $D$, see Proposition \ref{prop:reeb}. The basic idea behind the proof of Theorem \ref{thm:invariants}, which we give in Section \ref{sec: chfromehrhart}, is to use the star subdivision of $D$ centered at $v$ and write the number of rational points in the interior of $D$ as a sum of the numbers of rational points in the interior of each simplex of maximal dimension appearing in the subdivision (the condition that $\nu$ is non-degenerate implies that there are no rational points in lower dimensional simplices). Roughly speaking, the number of rational points in each simplex is related to the contributions to contact homology of the multiples of the simple Reeb orbit associated to the corresponding facet of $D$. 

Theorem~\ref{thm:invariants} implies that, given $j \in \frac{1}{m}\ZZ \cap \left]-1,0\right]$, we have
\begin{equation} \label{eq:invariants}
cb_{2j + 2a} (D) = \sum_{k=0}^{a} \delta_{m(n-j-k)} (D) \,,\ a \in \NN_0\,.
\end{equation}
It also has the following combinatorial consequences (cf. Theorem~\ref{thm:combinatorial} and Remark~\ref{rem:combinatorial}).
\begin{corollary} 
\quad
\begin{itemize}
\item[(1)] 
\[
cb_0 (D) = \#\left(\inte (mD)\cap \ZZ^n\right) \,.
\]
\item[(2)]
\[
cb_{2(n-1)} (D) = n! \vol (mD) -1 \,.
\]
\item[(3)]
\[
cb_{2j+2a} (D) = n! \vol (mD)\,,\ \forall\ j \in \frac{1}{m}\ZZ \cap \left]-1,0\right]\,,\ a\geq n \,.
\]
\end{itemize}
\end{corollary}

\noindent In subsection~\ref{ex:torsion} we give an example of Conley-Zehnder index and Ehrhart quasi-polynomial
computations, illustrating well the non-trivial content of Theorem~\ref{thm:invariants} and its consequences.

Theorem~\ref{thm:invariants} implies that the contact Betti numbers $cb_j (D)$ are toric contact invariants
of $(M_D^{2n+1}, \xi_D)$. In fact, one should be able to remove the word ``toric" and say that each contact Betti 
number $cb_j (D)$ is a contact invariant of $(M_D^{2n+1}, \xi_D)$, the rank of its degree $j$ cylindrical
contact homology $HC_j (M_D, \xi_D)$. Unfortunately, and despite recent foundational developments 
(e.g.~\cite{pardon2016,pardon2019}), cylindrical contact homology has not been proved to be a well defined 
invariant in the presence of contractible closed Reeb orbits, even in this restricted context of $\QQ$-Gorenstein 
toric contact manifolds. Hence, when we write $HC_j (M_D, \xi_D)$ in this paper we are just using a suggestive
notation for $\QQ^{cb_j(D)}$, i.e.
\[
HC_j (M_D, \xi_D) := \QQ^{cb_j(D)}\,.
\]
However, note that there are at least two particular contexts where this is more than suggestive notation:
\begin{itemize}
\item[(ii)] Gorenstein toric contact manifolds that have a non-degenerate toric contact form with all 
of its closed contractible Reeb orbits having Conley-Zehnder index strictly greater than $3-n$, i.e. 
contact homology degree strictly greater than $1$.
\item[(i)] Gorenstein toric contact manifolds that have crepant (i.e. with zero first Chern class) toric 
symplectic fillings.
\end{itemize}
Indeed, in both of these contexts we can use positive equivariant symplectic homology to conclude 
that the contact Betti numbers are contact invariants. For (i), one considers the
positive equivariant symplectic homology of the symplectization and the work of Bourgeois and 
Oancea~\cite[section 4.1.2]{oancea2014linearized}. For (ii), one considers the positive equivariant symplectic homology 
of the filling and recent work of McLean and Ritter~\cite{mclean2018mckay}, which uses previous work 
of Kwon and van Koert~\cite{kwonkoert} (see~\cite[Remark 1.4]{abreu2018contact} and~\cite[section 2]{abreu2019}). 

As we discuss in detail in section~\ref{resolutions}, a crepant toric symplectic filling of a $\QQ$-Gorenstein toric contact
manifold $(M_D, \xi_D)$ is constructed using the following two ingredients:
\begin{itemize}
\item[(i)] A triangulation $\T$ of $D$ and its corresponding fan $\Sigma$.
\item[(ii)] A strictly convex support function $\varphi$ on $\Sigma$ (cf. Definition~\ref{def:support}).
\end{itemize}
This gives rise to a crepant toric symplectic filling $(X_\Sigma, \omega_\varphi)$ of $(M_D, \xi_D)$, where $X_\Sigma$
is the toric variety associated to $\Sigma$ and $\omega_\varphi$ is a toric symplectic form on $X_\Sigma$ determined
by $\varphi$ (cf. Proposition~\ref{prop:symplectic}). $X_\Sigma$ is smooth whenever $\T$ is a unimodular triangulation
of $D$, which necessarily implies that $D$ has order $m=1$, otherwise it has orbifold singularities. In fact, as stated and proved in section~\ref{resolutions} 
(cf. Proposition~\ref{prop:filling}), every compact $\QQ$-Gorenstein toric contact manifold $(M_D, \xi_D)$ admits a not-necessarily
smooth crepant symplectic filling by a toric orbifold $(X_\Sigma, \omega_\varphi)$.

The Chen-Ruan cohomology of $X_\Sigma$ is a $\QQ$-graded ring $H^\ast_\textup{orb}(X_\Sigma)$ (cf.~\cite{chen2004new} 
and appendix~\ref{appendix}). Batyrev-Dais~\cite{batyrev1996strong} (in the smooth case) and Stapledon~\cite{stapledon} (cf. appendix~\ref{appendixB}) showed that
\[
\dim H^{2j}_{\textup{orb}}(X_\Sigma; \QQ)=\delta_{mj} \ \text{for}\ j \in \frac{1}{m} \NN_{0}\,,
\]
and the remaining orbifold cohomology groups are trivial. This can be combined with Theorem~\ref{thm:invariants} to give
the following result (cf. Corollary~\ref{isostapledon}).
\begin{corollary} \label{intro_isostapledon}
\[
cb_{2j} (D) =  \sum_{k\geq 0}  \dim H^{2n-2j+2k}_\textup{orb}(X_\Sigma; \QQ)\,,\ \forall j\in\QQ\,.
\]
\end{corollary}
\begin{remark}
For $m=1$ and smooth $X_\Sigma$ we provide in subsection~\ref{ss:isostapledon} a direct symplectic proof of this result
using symplectic (co)homology of $X_\Sigma$ and its positive/negative and $S^1$-equivariant versions.
\end{remark}
\begin{remark}
McLean and Ritter~\cite{mclean2018mckay} have a similar result for isolated finite quotient singularities, which
overlaps with Corollary~\ref{intro_isostapledon} when $M_D$ is a lens space, i.e. when $D$ is a simplex.
\end{remark}

Given an integral toric diagram $D\subset\RR^n$ and corresponding Gorenstein toric contact manifold
$(M_D^{2n+1}, \xi_D)$, take a rational point $v=w/r\in\QQ^n$ in the interior of $D$, with $w\in \ZZ^n$ and $r\in \ZZ^+$. Let $\nu=(w,r)\in \ZZ^{n+1}$ 
and consider the corresponding
contact form $\alpha_\nu$ and toric Reeb vector field $R_\nu$. The Reeb flow of $R_\nu$ is periodic 
and induces an almost free $S^1$-action on $M_D$. The quotient is a monotone compact symplectic
toric orbifold $(B_\nu := M_D/S^1, \omega_\nu)$, where $\pi^\ast \omega_\nu = d\alpha_\nu$ with
$\pi : M_D \to B_\nu$ the quotient projection. In other words, $(M_D, \xi_D)$ is the prequantization of
$(B_\nu, \omega_\nu)$. While monotone compact symplectic toric manifolds are in 1-1 correspondence 
with Delzant polytopes $\Delta$ such that $r\Delta$ is a reflexive polytope (up to translation) for 
some $r\in\RR^+$, monotone compact symplectic toric orbifolds $(B, \omega)$ are in 1-1 correspondence 
with labelled simple rational polytopes $\Delta$ such that $r\Delta$ is an almost-reflexive polytope (up 
to translation) for some $r\in\RR^+$. When $B=B_\nu$ for some $\nu$ as above,
we have that $r=r_\nu \in\NN$ and when $B$ is smooth we have that $c_1 (TB) = r [\omega] \in H^2 (B;\ZZ)$.
See section~\ref{baseorbifold} for more details regarding these facts.

Our third main result in this paper shows how the contact Betti numbers of a Gorenstein $(M_D, \xi_D)$ 
are determined by the Chen-Ruan (orbifold) cohomology of any such quotient $(B_\nu, \omega_\nu)$, once 
one takes into account its decomposition by twisted sectors and the corresponding degree shifting numbers.  
See appendix~\ref{appendix} for details regarding Chen-Ruan cohomology of toric orbifolds, twisted sectors 
and degree shifting numbers. 

Let 
\begin{equation}\label{eq: baseorbifolddecompositionintro}H_{orb}^\ast(B_\nu; \QQ)=\bigoplus_{0<T\leq 1} F_T^\ast(B_\nu)
\end{equation}
be the decomposition of the orbifold cohomology of $B_\nu$ obtained from representing $B_\nu$ as $M/S^1$ 
where the action of $S^1$ is induced by the flow of $-R_\nu$ (note that, although $R_\nu$ and $-R_\nu$ have 
the same periodic orbits up to orientation, they have different degree shifting numbers, with the later being the appropriate ones in 
this context, cf. Lemma~\ref{lem: twistedsectorsbox}). The summand $F_T^\ast(B_\nu)$ is the contribution of 
the fixed points of $e^{2\pi iT}\in S^1$ 
to the orbifold cohomology. Note in particular that $F_1(B_\nu)=H^\ast(B_\nu; \QQ)$ is the singular cohomology of $B_\nu$.

\begin{theorem} \label{thm:quotients}
Let $(M_D^{2n+1}, \xi_D)$ be a Gorenstein toric contact manifold, $w/r\in\QQ^n \cap \inte (D)$ with corresponding toric Reeb vector $\nu=(w,r)$, and
$(B_\nu = M_D/S^1, \omega_\nu)$ be as described above. Consider the Chen-Ruan 
cohomology of $B_\nu$ and its decomposition \eqref{eq: baseorbifolddecompositionintro}.

Then the contact homology of $(M_D, \xi_D)$ is given as a graded vector space by
\begin{equation} \label{eq:quotients}
\textup{HC}_\ast (M_D, \xi_D) = \bigoplus_{k\geq 0} \ \ \bigoplus_{0<T\leq 1} 
F_T^{\ast-2r_\nu T+2-2r_\nu k} (B_\nu)  \,.
\end{equation}
In particular, the right hand side is independent of $w/r\in\QQ^n \cap \inte (D)$.
\end{theorem}

\begin{remark}
Although a result of this type should also hold for $\QQ$-Gorenstein toric contact manifolds, it seems to
be harder to formulate precisely and we will not do it in this paper.
\end{remark}

When $B_\nu$ is smooth there is only the $T=1$ (non-)twisted sector with $(B_\nu)_1 = B_\nu$.
Hence, we get as a corollary the following result:

\begin{corollary} \label{cor:Bourgeois}
When $B_\nu$ is smooth we have that
\[
\textup{HC}_\ast (M_D, \xi_D) = \bigoplus_{k\geq 0} H^{\ast - 2r_\nu k - 2 (r_\nu -1)} (B_\nu ; \QQ)\,.
\]
\end{corollary}

Since $\textup{HC}_\ast (M_D, \xi_D)$ can only be different from zero when $\ast$ is even, when $r_\nu = 1$
we get from~(\ref{eq:quotients}) that
\[
F_T^q (B_\nu) \ne 0 \quad\Rightarrow\quad q = \,\text{even}\,+ 2 (1-T)\,.
\]
Hence,
\[
\bigoplus_{0<T\leq 1}  F_T^{\ast-2k+2(1-T)} (B_\nu) =
\bigoplus_{\ast -2k \leq q < \ast -2k+2} \ \  \bigoplus_{0<T\leq 1}  F_T^{q} (B_\nu) =
\bigoplus_{\ast -2k \leq q < \ast -2k+2} H_{orb}^q (B_\nu; \QQ)\,.
\]
Considering a modified grading in the orbifold cohomology of $B_\nu$ by rounding the degree down to the
nearest even integer (note that the orbifold cohomology is in general $\QQ$-graded), that is
\[
H_{\lfloor orb \rfloor}^{2j} (B_\nu;\QQ) := \bigoplus_{2j\leq q < 2j+2} H_{orb}^q (B_\nu)\,,
\]
we then get the following corollary of Theorems~\ref{thm:invariants} and~\ref{thm:quotients}.

\begin{corollary} \label{cor:Stapledon}
When $r_\nu =1$ we have that
\[
\textup{HC}_\ast (M_D, \xi_D) = \bigoplus_{k\geq 0} H_{\lfloor orb \rfloor}^{\ast - 2k} (B_\nu ; \QQ)
\]
and (cf. Theorem 4.3 in~\cite{stapledon})
\[
\delta_{n-k} (D) = \dim H_{\lfloor orb \rfloor}^{2k} (B_\nu ; \QQ) =
\sum_{2k\leq q < 2k+2} \dim H_{orb}^q (B_\nu; \QQ)\,,\ k\in\ZZ\,.
\]
In particular, the right hand sides are independent of $\nu$.
\end{corollary}

We remark that Theorem \ref{thm:quotients} and its corollaries should be interpreted as a combinatorial 
Morse-Bott theorem for contact homology \cite{bourgeois2002morse} in our toric setting. When we consider 
a Reeb vector $\nu$ with integer coordinates, the corresponding contact form is Morse-Bott in the sense of 
\cite[Definition 1.7]{bourgeois2002morse}. The critical manifolds are $M^T/S^1$ for every $T>0$, where 
$M^T$ is the fixed point set of $e^{2\pi iT}\in S^1$ . On the other hand, the inertia orbifold of $B$ is precisely 
the union of $M^T/S^1$ with $T\in ]0,1]$, so a relation between the orbifold cohomology of $B$ and the contact 
homology of $M$ is expected. Note that Corollary \ref{cor:Bourgeois} essentially appears in Bourgeois' thesis 
\cite[Proposition 9.1]{bourgeois2002morse}.

To avoid the aforementioned issues of contact homology not being proven to be a well-defined invariant, our proof of Theorem \ref{thm:quotients} will be entirely combinatorial, relying on Theorem~\ref{thm:invariants}. As remarked previously, the fact that we can relate the Ehrhart series of $D$ to the contact Betti numbers $cb_j(D, \nu)$ of a non-degenerate Reeb $\nu$ is explained by summing the counts of rational points in each simplex of maximal dimension in the star subdivision centered at $v$. When $v=w/r$ is a rational point, we can still calculate the Ehrhart series of $D$ by considering the star subdivision centered at $v$, but now counting rational points in every simplex and not just the maximal ones, cf. Lemma \ref{lem: ehrhartstartriangulation}. This leads to a connection between the Ehrhart series of $D$ and the orbifold cohomology of $B$ that we will use in the proof of Theorem \ref{thm:quotients}.
\subsection{Example}

Let us now see how these results apply  to an illustrative Gorenstein ($m=1$) example: the $5$-dimensional lens space
\[
L^5_3 (1,1,1) := (S^5/\ZZ_3, \xi_{std})\,.
\]
Its toric diagram $D\subset\RR^2$ is given by the simplex
\[
D = \text{conv} \, ((1,0), (0,1), (-1,-1)),
\]
with Ehrhart polynomial
\[
L_D (t) = \frac{1}{2} (3t^2 + 3t + 2) = \frac{(t+2)(t+1)}{2} + \frac{(t+1)t}{2} + \frac{t (t-1)}{2}\,.
\]
Hence,
\[
\delta_0 (D) = \delta_1 (D) = \delta_2 (D) = 1
\]
and we get from Theorem~\ref{thm:invariants} that the contact Betti numbers of $L^5_3 (1,1,1)$ are
\[
cb_\ast (L^5_3 (1,1,1))  =
\begin{cases}
1 &\textup{ if }\ast=0,\\
2 &\textup{ if }\ast=2,\\
3 &\textup{ if }\ast=\textup{even}\geq4,\\
0 &\textup{ otherwise.}
\end{cases}
\]

The triangulation of $D$ given by $\T = D$ corresponds to the crepant orbifold filling of
$L^5_3 (1,1,1) = S^5 / \ZZ_3$ by the corresponding quotient of the ball $B^6 / \ZZ_3$,
with an isolated orbifold singularity at the origin. In fact, for this triangulation we have that
$X_\Sigma = \CC^3 / \ZZ_3$ and
\[
H^\ast_{orb} (\CC^3 / \ZZ_3 ; \QQ) = F_1^\ast (\CC^3 / \ZZ_3) \oplus F_{1/3}^\ast (\CC^3 / \ZZ_3) 
\oplus F_{2/3}^\ast (\CC^3 / \ZZ_3)
\]
with
\[
F_1^\ast (\CC^3 / \ZZ_3) = H^\ast (pt)\,,\ F_{1/3}^\ast (\CC^3 / \ZZ_3) = H^{\ast-2} (pt) \ \text{and} \ 
F_{2/3}^\ast (\CC^3 / \ZZ_3) = H^{\ast-4} (pt)\,.
\]
Hence
\[
H_{orb}^\ast (\CC^3 / \ZZ_3; \QQ) =
\begin{cases}
\QQ &\textup{ if }\ast=0,2,4,\\
0 &\textup{ otherwise.}
\end{cases}
\]
Using Corollary~\ref{intro_isostapledon} we then have the following table giving the contributions
of $F_1$, $F_{1/3}$ and $F_{2/3}$ to the rank of $HC_\ast (L^5_3 (1,1,1) )$:
\begin{equation*}
\begin{array}{|c|c|c|c|c|c|cl} \hline
  \ast =  & \quad   0 \quad  & \quad 2 \quad & \quad 4 \quad & \quad 6 \quad & \quad 8 \quad & 
  \cdots
\\ 
\hline
F_1   &     0      &   0    &   1  &  1 & 1 & \cdots
\\ 
\hline 
F_{1/3}   &     0      &    1    &   1  &  1 & 1 & \cdots
\\ 
\hline
F_{2/3}  &     1      &    1    &   1  &  1 & 1 & \cdots
\\ 
\hline 
cb_\ast (L^5_3 (1,1,1)) &   1   & 2 & 3 & 3 & 3 & \cdots
\\
\hline
\end{array}
\end{equation*}

The triangulation of $D$ given by its barycentric subdivision, with $(0,0)$ the barycenter, is unimodular
and corresponds to the crepant smooth filling given by $X_\Sigma =$ total space of the line bundle $\Oo (-3)$ 
over $\CC P^2$. Hence
\[
H_{orb}^\ast (X_\Sigma ; \QQ) = H^\ast (X_\Sigma ; \QQ) = H^\ast (\CC P^2 ; \QQ) = 
\begin{cases}
\QQ &\textup{ if }\ast=0,2,4,\\
0 &\textup{ otherwise,}
\end{cases}
\]
and one can again use Corollary~\ref{intro_isostapledon} to determine $HC_\ast (L^5_3 (1,1,1) )$.
Note that this $X_\Sigma$ is the resolution of $\CC^3 / \ZZ_3$ and, as expected,
\[
H_{orb}^\ast (\CC^3 / \ZZ_3 ; \QQ) \cong H^\ast (X_\Sigma ; \QQ)\,.
\]

The total space of the line bundle $\Oo(-3)$ over $\CC P^2$ also describes how $L^5_3 (1,1,1)$ arises
as the prequantization of $(\CC P^2, 3\omega_{FS})$, corresponding to the choice of $\nu = (0,0)$,
$B_{(0,0)} = \CC P^2$ and $r_{(0,0)}=1$. Hence, in this case, both Corollary~\ref{cor:Bourgeois}  (with $r_\nu = 1$)
and Corollary~\ref{cor:Stapledon} (with $B_\nu$ smooth) determine $HC_\ast (L^5_3 (1,1,1) )$ in exactly the same
way as Corollary~\ref{intro_isostapledon}.

To illustrate an application of Theorem~\ref{thm:quotients} in this example, consider 
\[
\nu = \left(-\frac{1}{2},-\frac{1}{2}\right) \in \inte (D)\,.
\]
As we will see in subsection~\ref{topexample}, we then have that
\[
B_\nu = \text{weighted projective space}\  \CC P^2 (4,1,1)\,,\ r_\nu = 2
\]
and there are twisted sectors for $T=1/4, \,2/4, \,3/4, \,1$, with
\[
(B_\nu)_{1/4} = (B_\nu)_{2/4} = (B_\nu)_{3/4} = \{ pt \} \quad\text{and}\quad
(B_\nu)_1 = \CC P^2 (4,1,1)\,.
\]
The corresponding degree shifting numbers are
\[
\iota_{k/4} = k/2\,,\ k=1,\,2,\,3, \quad\text{and}\quad \iota_1 = 0\,.
\]
Hence, we have that
\[
F_{k/4}^\ast (B_\nu) = 
\begin{cases}
\QQ &\textup{ if }\ast=k,\\
0 &\textup{ otherwise,}
\end{cases}
\quad\text{for} \ k=1,\,2,\,3,\quad
F_1^\ast (B_\nu) = 
\begin{cases}
\QQ &\textup{ if }\ast=0,\,2,\,4,\\
0 &\textup{ otherwise,}
\end{cases}
\]
and
\[
H_{orb}^\ast (B_\nu = \CC P^2 (4,1,1); \QQ) = 
\begin{cases}
\QQ &\textup{ if }\ast=0,\,1,\,3,\,4,\\
\QQ \oplus \QQ & \textup{ if }\ast=2,\\
0 &\textup{ otherwise.}
\end{cases}
\]
Using Theorem~\ref{thm:quotients} we then have the following table giving the contributions of
$F_1$, $F_{1/4}$, $F_{2/4}$ and $F_{3/4}$ to the rank of $HC_\ast (L^5_3 (1,1,1) )$:
\begin{equation*}
\begin{array}{|c|c|c|c|c|c|c|cl} \hline
  \ast =  & \quad   0 \quad  & \quad 2 \quad & \quad 4 \quad & \quad 6 \quad & \quad 8 \quad & \quad 10 \quad & \cdots
\\ 
\hline
F_1   &     0      &   1    &   1  &  2 & 1 & 2 & \cdots
\\ 
\hline 
F_{1/4}   &     1      &    0    &   1  &  0 & 1 & 0 & \cdots
\\ 
\hline
F_{2/4}  &     0      &    1    &   0  &  1 & 0 & 1 & \cdots
\\ 
\hline 
F_{3/4}  &     0      &    0    &   1  &  0 & 1 & 0 & \cdots
\\ 
\hline 
cb_\ast (L^5_3 (1,1,1)) &   1   & 2 & 3 & 3 & 3 & 3 &  \cdots
\\
\hline
\end{array}
\end{equation*}

\section{Polytope geometry}

A ($n$-dimensional) convex polytope $\Delta$ is the convex hull of a finite set of points in $\RR^n$ 
(with non-empty interior). A polyhedral set is a subset of $\RR^n$ given by the intersection of a finite 
number of halfspaces. Note that a polytope is the same thing as a compact polyhedral set. 
We say that $F$ is a face of $\Delta$ if $F\subseteq \partial \Delta$ and $F$ is the intersection of $\Delta$ 
with some hyperplane. We say that a $0$, $1$ or $n-1$ dimensional face is a vertex, an edge or a facet, 
respectively.

\subsection{Ehrhart quasi-polynomial}

The Ehrhart quasi-polynomial of a rational polytope counts the number of rational points in the given polytope. 
Let $\Delta\subseteq \RR^n$ be such a polytope and let $m$ be a positive integer such that $m\Delta$ is 
an integral polytope, that is, $m\Delta$ has vertices in $\ZZ^n$. We call the minimal such $m$ the order of $D$.

We define, for $t\in \NN$,
$$L_\Delta(t)=\#\left(\Delta\cap \frac{1}{t}\ZZ^n\right).$$
Equivalently, $L_\Delta(t)$ is the number of integral points in $t\Delta$. The fact that this is a quasi-polynomial 
function and the following properties are well known.

\begin{theorem}
Given a rational polytope $\Delta\subseteq \RR^n$, $$L_\Delta(t)=\#\left(\Delta\cap \frac{1}{t}\ZZ^n\right)$$ 
is given by a quasi-polynomial function of period $m\in\NN$ for $t\in \NN$, called the Ehrhart quasi-polynomial. 
Moreover the Ehrhart quasi-polynomial has the following properties:
\begin{enumerate}
\item Each branch $t\mapsto L_\Delta(mt+j)$, $j=0,\ldots,m-1$, is a polynomial of degree $n$ and the leading 
term is $\vol(m\Delta)$.
\item The constant term is $L_\Delta(0)=1$ (if $\Delta$ is convex).
\item $L_\Delta(-t)=(-1)^n L_{\inte \Delta}(t)$ (Ehrhart reciprocity).
\end{enumerate}
\label{Ehrhart}
\end{theorem}

\noindent Proofs of these facts can be found in \cite{continuousdiscrete}: Theorem 3.23 proves quasi-polynomiality, the leading term is provided by Corollary 3.20 (when $m=1$) and Exercise 3.34 (general $m$), the constant term by  Corollary 3.15 ($m=1$) and Exercise 3.32 (general $m$), and Ehrhart reciprocity is Theorem 4.1. The reader can also see \cite{haase2012lecture}.

When $\Delta$ is an integral polytope,
i.e. $m=1$, the quasi-polynomial is indeed a polynomial. Note that the properties of the Ehrhart (quasi-)polynomial 
can be regarded as a generalization of Pick's theorem. Indeed, for an integral (i.e., $m=1$) polygon $\Delta$ in dimension $n=2$, 
the Ehrhart polynomial has degree 2 and since its leading term is the area of the polygon we have that
$$2\vol(\Delta) = L_\Delta(1)+L_\Delta(-1)-2L_\Delta(0)=L_\Delta(1)+L_\Delta(-1)-2\,,$$
with $L_\Delta(1)$ and $L_\Delta(-1)$ being the number of integral points in $\Delta$ and $\inte \Delta$, respectively.

We can also associate to the polytope its Ehrhart series, which is the generating series of the Ehrhart quasi-polynomial:
$$\Ehr_\Delta(z)=\sum_{t=0}^\infty L_\Delta(t)z^t.$$
The fact that $L_\Delta$ is a quasi-polynomial of period $m$ implies that the Ehrhart series can be written in the form
\begin{equation}\Ehr_\Delta(z)=\frac{1}{(1-z^m)^{n+1}}\left(\sum_{j=0}^{m(n+1)-1} \delta_j z^j\right)
\label{ehrseries}
\end{equation}
for some coefficients $\delta_j=\delta_j(\Delta)$, $j=0, \ldots, m(n+1)-1$. The vector $(\delta_0, \ldots, \delta_{m(n+1)-1})$ 
is known as the $\delta$-vector of the polytope $\Delta$ (and it is also called the $h^\ast$-vector by some authors) and the 
numerator of the Ehrhart series is called the $\delta$-polynomial. It is a result by Stanley \cite[Theorems 1.6, 1.7]{stanley} that, in fact, $\delta_j$ are non-negative 
integers. The Ehrhart polynomial can be recovered from the $\delta$-vector by
\begin{equation}L_\Delta(t)=\sum_{j\equiv t \mod m} \delta_j \binom{\frac{t-j}{m}+n}{n}.
\label{Ehrhartbasis}
\end{equation}

Note that by the Ehrhart reciprocity we have:
\begin{align*}L_{\inte \Delta}(t+m)&=(-1)^n L_\Delta(-t-m)=(-1)^n\left(\sum_{j\equiv -t \mod m} \delta_j \binom{\frac{-t-m-j}{m}+n}{n}\right)\\
&=\sum_{j\equiv -t \mod m} \delta_j \binom{\frac{t+j}{m}}{n}=\sum_{j \equiv t \mod m} \delta_{mn-j} \binom{\frac{t-j}{m}+n}{n}. 
\end{align*}

Hence, we define for convenience
\begin{equation}
\Ehr_{\inte\Delta}(z)\equiv \sum_{t=0}^\infty L_{\inte\Delta}(t+m)z^t=\frac{1}{(1-z^m)^{n+1}}\left(\sum_{j=0}^{m(n+1)-1} \delta_{mn-j} z^j\right).
\label{ehrinterior}
\end{equation}

\subsection{Subdivisions of polytopes}

We define here subdivisions of polytopes and some notation associated with them. 

\begin{definition}
A subdivision of a polytope $\Delta$ is a set of polytopes $\T=\{\theta\}$ with the following properties:
\begin{enumerate}
\item If $\theta'$ is a face of $\theta\in \T$ then $\theta'\in \T$.
\item If $\theta_1, \theta_2\in \T$ then their intersection is a (possibly empty) common face of $\theta_1$ and $\theta_2$.
\item The family of polytopes $\T$ covers $\Delta$, that is, $\Delta=\bigcup_{\theta\in \T}\theta$.
\end{enumerate}
Let $\T_d$ denote the set of $d$-dimensional polytopes in $\T$. We say that a subdivision is rational if 
$m\T_0\subseteq \ZZ^n$ where $m$ is the order of $\Delta$; we will always assume that our subdivisions are rational. If every 
polytope in $\T$ is a simplex we say that $\T$ is a triangulation. 
\end{definition}

Rational triangulations are relevant in our context since such triangulations of a toric diagram $D$ correspond to a (toric) crepant resolution of 
the symplectic cone $W$, as we will explain in section \ref{resolutions}. The corresponding resolution is smooth only when $\T$ is a 
unimodular triangulation, that is, a subdivision $\T$ such that $\theta$ 
is an integral, unimodular simplex for every $\theta\in \T$. This implies in particular that $D$ is integral, i.e. $m=1$. Note that for $n>2$ 
not every integral polytope of dimension $n$ admits a unimodular triangulation, but when $n\leq 2$ it always does. 

\subsection{Reflexive polytope}

The dual of a polytope $\Delta\subseteq \RR^n$ (also called polar polytope) is 
$$\Delta^\circ=\{y\in \RR^n: \langle x, y\rangle \geq -1\textup{ for all }x\in \Delta\}.$$
Note that $(\Delta^\circ)^\circ=\Delta$. Note that the vertices of the dual polytope correspond to the vectors normal to the 
facets of the original polytope, and vice-versa. More generally, there is a correspondence between $d$-faces of $\Delta$ 
and $(n-d)$-faces of $\Delta^\circ$. 

\begin{definition}
An integral polytope $\Delta$ is said to be reflexive if $\Delta^\circ$ is also an integral polytope. 
\end{definition}
An integral polytope is reflexive if and only if it can be written in the form
$$\Delta=\{x\in \RR^n: \langle x, \nu_j\rangle\geq -1, \,j=1, \ldots, d\}.$$

Up to $\textup{Aff}(n, \ZZ)$ equivalence, there is a finite number of reflexive polytopes for each $n$. For instance for $n=2$ 
there are precisely 16. 

An interesting result concerning the Ehrhart polynomials of reflexive polytopes is the following:
 
\begin{theorem}[Hibi's Palindromic Theorem, \cite{hibi1990some}]\label{hibbi}
An integral polytope $\Delta$ is reflexive if and only if its $\delta$-vector is palindromic, that is, $\delta_j=\delta_{n-j}$ for 
$j=0, \ldots, n$.
\end{theorem}

We remark that by equations \eqref{ehrseries} and \eqref{ehrinterior} the palindromic condition is equivalent to 
$$\# [t\Delta\cap \ZZ^n]=\#[(t+1)\inte\Delta\cap \ZZ^n] \textup{ for every }t\in \NN.$$

\section{Gorenstein Toric Contact Manifolds}
\label{toriccontact}

In this section we will explain the $1$-$1$ correspondence between $\QQ$-Gorenstein toric contact manifolds, 
i.e. good toric contact manifolds (in the sense of Lerman~\cite{lerman2002contact}) with 
torsion first Chern class, and rational toric diagrams (defined below). We extend the presentation in~\cite{abreu2016mean} 
which is only between Gorenstein contact manifolds and integral toric diagrams.

Via symplectization, there is a $1$-$1$ correspondence between co-oriented contact manifolds
and symplectic cones, i.e. triples $(W,\om,X)$ where $(W,\om)$ is a connected symplectic manifold
and $X$ is a vector field, the Liouville vector field, generating a proper $\R$-action
$\rho_t:W\to W$, $t\in\R$, such that $\rho_t^\ast (\om) = e^{t} \om$. A closed symplectic cone is a 
symplectic cone $(W,\om,X)$ for which the corresponding contact manifold $M = W/\R$ is closed.

A toric contact manifold is a contact manifold of dimension $2n+1$ equipped with an effective Hamiltonian
action of the standard torus of dimension $n+1$: $\TT^{n+1} = \RR^{n+1} / 2\pi\ZZ^{n+1}$. Also via symplectization,
toric contact manifolds are in $1$-$1$ correspondence with toric symplectic cones, i.e. symplectic cones
$(W,\om,X)$ of dimension $2(n+1)$ equipped with an effective $X$-preserving Hamiltonian $\TT^{n+1}$-action,
with moment map $\mu : W \to \R^{n+1}$ such that $\mu (\rho_t (w)) = e^{t} \mu (w)$, for all $w\in W$ and $t\in\R$.
Its moment cone is defined to be $C:= \mu(W) \cup \{ 0\} \subset \R^{n+1}$.

A toric contact manifold is {\it good} if its toric symplectic cone has a moment cone with the following properties.
\begin{definition} \label{def:good}
A cone $C\subset\R^{n+1}$ is \emph{good} if it is strictly convex and there exists a minimal set 
of primitive vectors $\nu_1, \ldots, \nu_d \in \ZZ^{n+1}$, with 
$d\geq n+1$, such that
\begin{itemize}
\item[(i)] $C = \bigcap_{j=1}^d \{x\in\R^{n+1}\mid 
\ell_j (x) := \langle x, \nu_j \rangle \geq 0\}$.
\item[(ii)] Any codimension-$k$ face of $C$, $1\leq k\leq n$, 
is the intersection of exactly $k$ facets whose set of normals can be 
completed to an integral basis of $\ZZ^{n+1}$.
\end{itemize}
The primitive vectors $\nu_1, \ldots, \nu_d \in \Z^{n+1}$ are called the defining normals of the good cone $C\subset\R^{n+1}$.
\end{definition}
The analogue for good toric contact manifolds of Delzant's classification theorem for closed toric
symplectic manifolds is the following result (see~\cite{lerman2002contact}).
\begin{theorem} \label{thm:good}
For each good cone $C\subset\R^{n+1}$ there exists a unique closed toric symplectic cone
$(W_C, \om_C, X_C, \mu_C)$ with moment cone $C$.
\end{theorem}
The existence part of this theorem follows from an explicit symplectic reduction of the standard euclidean
symplectic cone $(\R^{2d}\setminus\{0\}, \omega_{\rm st}, X_{\rm st})$, where $d$ is the number of defining normals of the
good cone $C\subset\R^{n+1}$, with respect to the action of a subgroup $K\subset\TT^d$ induced by the standard
action of $\TT^d$ on $\R^{2d}\setminus\{0\} \cong \CC^d \setminus\{0\}$. More precisely,
\begin{equation} \label{eq:defK}
K := \left\{[y]\in\TT^d \mid \sum_{j=1}^d y_j \nu_j \in 2\pi\Z^{n+1}   \right\}\,,
\end{equation}
where $\nu_1, \ldots, \nu_d \in \Z^{n+1}$ are the defining normals of $C$, i.e. $K:= \ker (\beta)$ where $\beta :\TT^d \to \TT^{n+1}$
is represented by the matrix
\begin{equation} \label{eq:defBeta}
\left[\ \nu_1 \ | \ \cdots \ | \ \nu_d \  \right]\,.
\end{equation}
Depending on the context, which will be clear in each case, we will also denote by $\beta$ the map from $\Z^d$ to $\Z^{n+1}$  
represented by this matrix.

The Chern classes of a co-oriented contact manifold can be canonically identified with 
the Chern classes of the tangent bundle of the associated symplectic cone.
The following proposition gives a moment cone characterization for whether the first Chern class is torsion; this result is commonly 
used in toric Algebraic Geometry (see, e.g., section $4$ of~\cite{batyrev1996strong}).
\begin{prop} \label{prop:c_1}
Let $(W_C, \omega_C,X_C)$ be a good toric symplectic cone.
Let $\nu_1,\ldots,\nu_d \in \Z^{n+1}$ be the defining normals of the corresponding moment cone 
$C\subset\R^{n+1}$. Then $m \,c_1 (TW_C) = 0$ if and only if there exists $\nu^\ast \in (\Z^{n+1})^\ast$ 
such that
\[
\nu^\ast (\nu_j) = m\,,\ \forall\ j=1,\ldots,d\,.
\]
\end{prop}
\begin{proof}
Let $D_i$ be the toric divisor associated to the facet of $C$ with normal $\nu_j$. The first Chern class is well-known to be given 
as the sum of the toric divisors \cite[Theorem 4.1.3]{cox}
$$c_1(TW_C)=\sum_{i=1}^d [D_i]\in H^2(W_C; \ZZ).$$
The second cohomology group is the cokernel of the map $\beta^t$
$$0\rightarrow \ZZ^{n+1} \overset{\beta^t}{\longrightarrow} \ZZ^d \longrightarrow H^2(W_C; \ZZ)\rightarrow 0$$
(see \cite[Theorem 8.2.3]{cox}). The map from $\ZZ^d$ to $H^2(W_C;\ZZ)$ sends a basis vector $e_i$ to the class of the divisor $D_i$. 
By the formula above, $c_1(TW_C)$ is the image of $(1,\ldots, 1)\in \ZZ^d.$ Hence $m \,c_1 (TW_C) = 0$ if and only if
$$(m, \ldots, m)^t =\beta^t \nu$$ for some $\nu \in \ZZ^{n+1}$. Under the identification $\ZZ^{n+1}\cong (\ZZ^{n+1})^\ast$, the previous
 condition translates to 
$\nu^\ast(\nu_j) = m\,,\ \forall\ j=1,\ldots,d$. \qedhere
\end{proof}

By an appropriate change of basis of the torus $\TT^{n+1}$, i.e. an appropriate $SL(n+1,\Z)$
transformation of $\R^{n+1}$, this implies the following.
\begin{cor} \label{cor:c_1}
Let $(W_C, \omega_C,X_C)$ be a good toric symplectic cone with $c_1 (TW_C)$ torsion. 
Let $m\in\NN$ be the order of $c_1 (TW_C)$, i.e. the minimal positive integer such that $m \,c_1 (TW_C) = 0$. 
Then there exists an integral basis of $\TT^{n+1}$ for which the defining normals $\nu_1,\ldots,\nu_d \in \Z^{n+1}$ 
of the corresponding moment cone $C\subset\R^{n+1}$ are of the form
\[
\nu_j = (\tilde v_j, m)\,,\ \tilde v_j \in \Z^n\,,\ j=1,\ldots,d\,.
\]
\end{cor}

When we are in the conditions of the last corollary, we will encode the $\QQ$-Gorenstein toric contact manifold 
whose symplectization is $(W_C, \omega_C, X_C)$ by the rational polytope
$$D=\conv(v_1, \ldots, v_d)\subseteq \RR^n$$
where $v_i=\tilde v_i/m\in \QQ^n$. The integer $m$ will be called the order of $D$ or the order of $M_D$.

The next definition and theorem are then the natural analogues for $\QQ$-Gorenstein toric contact
manifolds of Definition~\ref{def:good} and Theorem~\ref{thm:good}.
\begin{defn} \label{def:diagram}
A \emph{(rational) toric diagram} $D\subset\R^n$ of order $m$ is a rational simplicial polytope with
all of its facets $\Aff(n,\Z)$-equivalent to $\conv\left(\frac{1}{m}e_1, \ldots, \frac{1}{m}e_n\right)$, where
$\{e_1,\ldots,e_n\}$ is the canonical basis of $\R^n$.
\end{defn}
\begin{rem} \label{rem:diagram}
The group $\Aff(n,\Z)$ of integral affine transformations of $\R^n$ can be naturally identified
with the elements of $SL(n+1,\Z)$ that preserve the hyperplane $\left\{ (mv,m)\mid v\in\R^n \right\} \subset\R^{n+1}$. 
The conditions imposed on toric diagrams $D=\conv(v_1, \ldots, v_d)$
are equivalent to the corresponding cone with normals $\nu_j=(mv_j, m)$, $j=1,\ldots,d$, being smooth.
\end{rem}
\begin{thm} \label{thm:diagram}
For each (rational) toric diagram $D\subset\R^n$ there exists a unique $\QQ$-Gorenstein toric contact
manifold $(M_D, \xi_D)$ of dimension $2n+1$.
\end{thm}
\noindent This correspondence specializes to \cite[Theorem 2.7]{abreu2016mean} when $m=1$. 

The $\TT^{n+1}$-action associates to every vector $\nu \in \R^{n+1}$ a contact vector field $R_\nu \in \Xx (M_D, \xi_D)$. 

\begin{defn} \label{def:reeb}
We will say that a contact form $\alpha_\nu \in \Omega^1 (M_D,\xi_D)$ is \emph{toric} if its Reeb vector field $R_{\alpha_\nu}$ 
satisfies
\[
R_{\alpha_\nu} = R_\nu \quad\text{for some $\nu\in\R^{n+1}$.}
\]
In this case we will say that $\nu\in\R^{n+1}$ is a \emph{toric Reeb vector} and that $R_\nu$ is a \emph{toric Reeb vector field}.

A \emph{normalized} toric Reeb vector is a toric Reeb vector $\nu\in\R^{n+1}$ of the form
\[
\nu = (mv,m) \quad\text{with $v\in\R^n$.}
\]
\end{defn}
\begin{prop}[\!{\!\!\cite{MSY} or \cite[Corollary 2.15]{abreu2012contact}}] \label{prop:reeb}
The interior of a toric diagram $D \subset \R^n$ parametrizes the set of normalized toric 
Reeb vectors on the $\QQ$-Gorenstein toric contact manifold $(M_D, \xi_D)$, i.e. $\nu = (mv,m)$ is
a normalized toric Reeb vector iff $v\in\textup{int} (D)$.
\end{prop}

\section{Conley-Zehnder index}

In this section we describe how the explicit method to compute the Conley-Zehnder index of any 
closed toric Reeb orbit on a Gorenstein toric contact manifold, described in~\cite[section 5]{abreu2012contact} 
and~\cite[section 3]{abreu2018contact}, also applies to $\QQ$-Gorenstein toric contact manifolds.

Given a toric diagram $D = \conv (v_1, \ldots, v_d) \subset \R^n$ of order $m\in\NN$ and corresponding 
$\QQ$-Gorenstein toric contact manifold $(M_D, \xi_D)$, consider a toric Reeb vector field $R_\nu \in \Xx (M_D, \xi_D)$ 
determined by the normalized toric Reeb vector (cf. Proposition~\ref{prop:reeb})
\[
\nu = (mv,m) \quad\text{with}\quad v = \sum_{j=1}^d a_j v_j\,,\ a_j\in\R^+\,,\ j=1, \ldots, d\,,\ \text{and}\ 
\sum_{j=1}^d a_j = 1\,.
\]
By a small abuse of notation, we will also write
\[
R_\nu = \sum_{j=1}^d a_j \nu_j \,,
\]
where $\nu_j = (mv_j,m)$, $j=1,\ldots,d$, are the defining normals of the associated good moment
cone $C\subset\R^{n+1}$.
Making a small perturbation of $\nu$ if necessary, we can assume that
\[
\text{the $1$-parameter subgroup generated by $R_\nu$ is dense in $\TT^{n+1}$,}
\]
which means that if $mv = (r_1, \ldots,  r_n)$ then $m,r_1,\ldots,r_n$'s are $\QQ$-independent.
This is equivalent to the corresponding toric contact form being non-degenerate. 
In fact, the toric Reeb flow of $R_\nu$ on $(M_D,\xi_D)$ has exactly $m$ simple closed
orbits $\gamma_1, \ldots,\gamma_m$, all non-degenerate, corresponding to the $m$ edges
$E_1,\ldots,E_m$ of the cone $C$, i.e. one non-degenerate closed simple toric $R_\nu$-orbit
for each $S^1$-orbit of the $\TT^{n+1}$-action on $(M_D,\xi_D)$. Equivalently, there is
\[
\text{one non-degenerate closed simple toric $R_\nu$-orbit for each facet of the toric diagram $D$.}
\]
Let $\gamma$ denote one of those non-degenerate closed simple toric $R_\nu$-orbits and assume 
without loss of generality that the vertices of the corresponding facet, necessarily a simplex, are 
$v_1, \ldots, v_n$. Let $h\in\R^n$ and $k\in\Z$ be such that
\[
\{ \nu_1 = (mv_1,m), \ldots, \nu_n = (mv_n,m), \eta = (mh, k)\} \ \text{is a $\Z$-basis of $\Z^{n+1}$.}
\]
Then $R_\nu$ can be uniquely written as
\[
R_\nu = \sum_{j=1}^n b_j \nu_j + b \eta\,,\ \text{with}\ b_1,\ldots, b_n\in \R\ \text{and}\ 
b = \frac{m}{k} \left(1 - \sum_{j=1}^n b_j \right) \ne 0\,.
\]

When $m=k=1$, as shown in~\cite[section 5]{abreu2012contact} and~\cite[section 3]{abreu2018contact},
the Conley-Zehnder index of $\gamma^{N}$, for any $N\in\NN$, is given by
\begin{equation} \label{eq:CZindex0}
\mu_{\rm CZ} (\gamma^{N}) = 2 \left( \sum_{j=1}^n \left\lfloor N \frac{b_j}{|b|}\right\rfloor
 + N \frac{b}{|b|} \sum_{j=1}^d \teta_j \right) + n\,,
\end{equation}
where $\teta\in\Z^d$ is such that
\[
\beta (\teta) = \eta - \frac{\beta (g)}{2\pi}\,,\ \text{for some $g\in K\cap SU(d)$,}
\]
where $K$ is given by~(\ref{eq:defK}) and the map $\beta: \Z^{d} \to \Z^{n+1}$ is defined 
by~(\ref{eq:defBeta}). This implies in particular that
\[
\sum_{j=1}^d \teta_j  = 1
\]
and so
\[
\mu_{\rm CZ} (\gamma^{N}) = 2 \left( \sum_{j=1}^n \left\lfloor N \frac{b_j}{|b|}\right\rfloor
 + N \frac{b}{|b|} \right) + n\,.
 \]

When $m>1$ we can still use formula~(\ref{eq:CZindex0}) to compute $\mu_{\rm CZ} (\gamma^{N})$.
In this case, we have that $\teta\in\Z^d$ is such that
\[
\beta (\teta) = \eta - \frac{\beta (g)}{2\pi}\,,\ \text{for some $g\in K\cap U(d)$,}
\]
but we might not be able to choose $g\in SU(d)$. This means that
\[
m \sum_{j=1}^d \teta_j  = k - r\,,
\]
where $r\in\Z$ is only defined up to a multiple of $m$. This ambiguity, coming from the ambiguity in
the choice of $g\in K$ or, equivalently, in the choice of closing path for the lifted orbit in $\CC^d$, can
be fixed as explained in~\cite{convex}. This implies that formula~(\ref{eq:CZindex0}) remains valid for any 
$m\in\NN$ by considering that
\[
\sum_{j=1}^d \teta_j  = \frac{k}{m}\,.
\]
Hence, we have that
\begin{equation} \label{eq:CZindexQ}
\mu_{\rm CZ} (\gamma^{N}) = 2 \left( \sum_{j=1}^n \left\lfloor N \frac{b_j}{|b|}\right\rfloor
 + N \frac{b}{|b|} \frac{k}{m} \right) + n\,.
\end{equation}

\subsection{Example}
\label{ex:torsion}

Consider the unit cosphere bundle of $S^3$ with its standard contact structure. This is a
Gorenstein toric contact manifold that can also be seen as the prequantization of $S^2 \times S^2$
with split symplectic form with area $2\pi$ on each $S^2$-factor. By considering prequantizations
of $S^2 \times S^2$ with split symplectic form with area $2\pi k$ on each $S^2$-factor, $k\in\NN$, 
we are looking at toric contact manifolds obtained as $\ZZ_k$ quotients of the unit 
cosphere bundle of $S^3$ with its standard contact structure. 

The first such quotient which is not Gorenstein is for $k=3$, i.e. the prequantization of 
$S^2 \times S^2$ with split symplectic form with area $6\pi$ on each $S^2$-factor. Viewed
this way, this toric contact manifold has moment cone with normals
\[
(1, 0, 0)\,,\ (0, 1, 0)\,,\ (-1, 0, 3)\ \text{and}\ (0, -1, 3)\,.
\]
The linear map given by the matrix
$$
\begin{bmatrix}
1 & 1 & 1\\ 
2 & 1 & 1 \\
3 & 3 & 2
\end{bmatrix}
\in SL(3,\ZZ)
$$
sends these normals to (up to ordering)
\[
\nu_1 = (1, 1, 3)\,,\ \nu_2 = (1, 2, 3)\,,\ \nu_3 = (2, 2, 3)\ \text{and}\ \nu_4 = (2, 1, 3)\,.
\]
Hence, we have the toric diagram $D = \conv (v_1, v_2, v_3, v_4)$ of order $m=3$, where
\[
v_1 = (1/3, 1/3)\,,\ v_2 = (1/3, 2/3)\,,\ v_3 = (2/3, 2/3)\ \text{and}\ v_4 = (2/3, 1/3)\,.
\]

Consider the normalized toric Reeb vector
\[
\nu = (1 + \varepsilon_1, 1 + \varepsilon_2, 3) \,,\  \text{with}\ 0 < \varepsilon_1 < \varepsilon_2\,,
\]
and denote by $\gamma_1, \gamma_2,\gamma_3$ and $\gamma_4$ its simple closed orbits corresponding to the facets
of $D$ with vertices $(v_1, v_2), (v_2,v_3), (v_3, v_4)$ and $(v_4,v_1)$, respectively.
With this $\nu$ and arbitrarily small $0 < \varepsilon_1 < \varepsilon_2$, we have that $\mu_{CZ} (\gamma_1^N)$ and
$\mu_{CZ} (\gamma_4^N)$ can be made arbitrarily large for any $N\in\NN$, hence the contact Betti numbers can
be determined by computing $\mu_{CZ} (\gamma_2^N)$ and $\mu_{CZ} (\gamma_3^N)$.

To compute $\mu_{CZ} (\gamma_2^N)$, note that $\nu_2$ and $\nu_3$ can be completed to a $\ZZ$-basis
with $\eta = (0, 1, k=1)$. We then have that
\[
\nu = (1 + \varepsilon_1, 1 + \varepsilon_2, 3) = (3 - \varepsilon_1 - 2\varepsilon_2) \nu_2 +
(-1 + \varepsilon_1 + \varepsilon_2) \nu_3 + (-3 + 3\varepsilon_2) \eta\,.
\]
It follows from~(\ref{eq:CZindexQ}) that
\begin{align*}
\mu_{CZ} (\gamma_2^N) & = 2 \left( \left\lfloor N \cdot \frac{3 - \varepsilon_1 - 2\varepsilon_2}{3 -3\varepsilon_2} \right\rfloor  
+  \left\lfloor N \cdot \frac{-1 + \varepsilon_1 + \varepsilon_2}{3 - 3\varepsilon_2} \right\rfloor - \frac{N}{3}\right) + 2  \\
& = 2 \left( \left\lfloor N \left(1 + \delta_1 \right) \right\rfloor  + \left\lfloor N \left(-\frac{1}{3} + \delta_2  \right)\right\rfloor - \frac{N}{3} \right) + 2\,,\ 
\text{for arbitrarily small $\delta_1, \delta_2 > 0$,}
\end{align*}
which means that the contribution of $\gamma_2$ and its iterates to the contact Betti numbers is
\[
1\ \text{in degrees} \ \frac{4}{3} \ \text{and}\ \frac{8+2k}{3}\,,\ k\in\NN_0\,.
\]

To compute $\mu_{CZ} (\gamma_3^N)$, note that $\nu_3$ and $\nu_4$ can be completed to a $\ZZ$-basis
with $\eta = (1, 0, k=1)$. We then have that
\[
\nu = (1 + \varepsilon_1, 1 + \varepsilon_2, 3) = (-1 + \varepsilon_1 + \varepsilon_2) \nu_3 +
(3 - 2 \varepsilon_1 - \varepsilon_2) \nu_4 + (-3 + 3\varepsilon_1) \eta\,.
\]
It follows from~(\ref{eq:CZindexQ}) that
\begin{align*}
\mu_{CZ} (\gamma_3^N) & = 2 \left( \left\lfloor N \cdot \frac{-1 + \varepsilon_1 +\varepsilon_2}{3 -3\varepsilon_1} \right\rfloor  
+  \left\lfloor N \cdot \frac{2 - 2 \varepsilon_1 - \varepsilon_2}{3 - 3\varepsilon_1} \right\rfloor - \frac{N}{3}\right) + 2  \\
& = 2 \left( \left\lfloor N \left(-\frac{1}{3} + \delta_1  \right)\right\rfloor + \left\lfloor N \left(1 - \delta_2 \right) \right\rfloor - \frac{N}{3} \right) + 2\,,\ 
\text{for arbitrarily small $\delta_1, \delta_2 > 0$,}
\end{align*}
which means that the contribution of $\gamma_3$ and its iterates to the contact Betti numbers is
\[
1\ \text{in degrees} \ -\frac{2}{3}\,,\ \frac{2}{3}\,,\ \frac{4}{3}\,,\ \frac{6}{3} \ \text{and}\ \frac{8+2k}{3}\,,\ k\in\NN_0\,.
\]

We conclude that
\[
cb_{2j/3} (D) =
\begin{cases}
1 &\textup{ if } j = -1,\,1,\,3\\
2  & \textup{ if } j = 2,\, 4+k,\, k\in\NN_0\\
0 &\textup{ otherwise.}
\end{cases}
\]

Let us now use this example to illustrate the content of Theorem~\ref{thm:invariants}. Since we already know the
contact Betti numbers, it gives us the $\delta$-vector of $D$:
\[
cb_{-8/3} (D) = 0\,,\ cb_{-2/3} (D) = 1 \ \text{and} \ cb_{4/3} (D) = 2 \Rightarrow \delta_7 (D) = \delta_4 (D) = 1\,,
\]
\[
cb_{-4/3} (D) = 0\,,\ cb_{2/3} (D) = 1 \ \text{and} \ cb_{8/3} (D) = 2 \Rightarrow \delta_5 (D) = \delta_2 (D) = 1\,,
\]
\[
cb_0 (D) = 0\,,\ cb_2 (D) = 1 \ \text{and} \ cb_4 (D) = 2 \Rightarrow \delta_3 (D) = \delta_0 (D) = 1
\]
and all other $\delta_k (D)$ are equal to zero. This means that the Ehrhart quasi-polynomial of 
\[
D = \conv ((1/3, 1/3), (1/2,2/3), (2/3,2/3), (2/3,1/3)) \subset \RR^2
\]
must be given by
\begin{align*}
t\equiv 1 \mod 3  \Rightarrow L_D (t) & = \delta_1 (D) \binom{\frac{t-1}{3}+2}{2} + \delta_4 (D) \binom{\frac{t-4}{3}+2}{2} + 
\delta_7 (D) \binom{\frac{t-7}{3}+2}{2} \\ 
& = \frac{1}{2} \left( \frac{t-4}{3} + 2 \right) \left( \frac{t-4}{3} + 1 \right) + \frac{1}{2} \left( \frac{t-7}{3} + 2 \right) \left( \frac{t-7}{3} + 1 \right) \\
& = \frac{1}{9} (t-1)^2\,,
\end{align*} 
\begin{align*}
t\equiv 2 \mod 3  \Rightarrow L_D (t) & = \delta_2 (D) \binom{\frac{t-2}{3}+2}{2} + \delta_5 (D) \binom{\frac{t-5}{3}+2}{2} + 
\delta_8 (D) \binom{\frac{t-8}{3}+2}{2} \\ 
& = \frac{1}{2} \left( \frac{t-2}{3} + 2 \right) \left( \frac{t-2}{3} + 1 \right) + \frac{1}{2} \left( \frac{t-5}{3} + 2 \right) \left( \frac{t-5}{3} + 1 \right) \\
& = \frac{1}{9} (t+1)^2\,,
\end{align*} 
\begin{align*}
t\equiv 0 \mod 3  \Rightarrow L_D (t) & = \delta_0 (D) \binom{\frac{t}{3}+2}{2} + \delta_3 (D) \binom{\frac{t-3}{3}+2}{2} + 
\delta_6 (D) \binom{\frac{t-6}{3}+2}{2} \\ 
& = \frac{1}{2} \left( \frac{t}{3} + 2 \right) \left( \frac{t}{3} + 1 \right) + \frac{1}{2} \left( \frac{t-3}{3} + 2 \right) \left( \frac{t-3}{3} + 1 \right) \\
& = \frac{1}{9} (t+3)^2\,,
\end{align*} 
which is indeed the case.

\section{Contact homology from Ehrhart theory}
\label{sec: chfromehrhart}

In this section we will prove our first main result, Theorem \ref{thm:invariants}, which establishes a relation between the contact homology of a toric contact manifold and the Ehrhart series of its toric diagram.

Take a toric diagram $D=\textup{conv}(v_1, \ldots, v_d)\subseteq \RR^n$ so that $mD$ is integral and let $(M, \xi)$ 
be the corresponding contact manifold. Consider a Reeb vector field determined by $\nu=(mv,m)$ with $v\in \inte D$ 
and $\QQ$-independent coordinates, as explained in Proposition \ref{prop:reeb}. Given such a Reeb, we define in Definition \ref{def:Betti_Euler} the contact Betti numbers $cb_j(D, \nu)$ by counting closed Reeb orbits with fixed degree. We restate Theorem \ref{thm:invariants} for the convenience of the reader and prove it. 

\begin{theorem}
Let $D\subseteq \RR^n$ be a toric diagram and let $(M, \xi)$ be the $m$-Gorenstein toric contact manifold associated to $D$. Then 
$$cb_{2j}(D, \nu)-cb_{2(j-1)}(D, \nu)=\delta_{m(n-j)}$$
where $(\delta_0, \ldots, \delta_{m(n+1)-1})$ is the $\delta$-vector of $\Delta$.  In particular, the contact Betti numbers $cb_{2j}(D, \nu)$ 
do not depend on the choice of toric Reeb vector $\nu$.
\label{main}
\end{theorem}
\begin{proof}
Given a facet $\ell$ of $D$ let $v_{\ell_1}, \ldots, v_{\ell_n}$ be the corresponding 
vertices and let $\eta_\ell=(m h_\ell, k)$ complete 
$$\nu_{\ell_1}=(mv_{\ell_1}, m), \ldots, \nu_{\ell_n}=(mv_{\ell_n}, m)$$ 
to a $\ZZ^{n+1}$-basis. Let $D_\ell=\textup{conv}(v, v_{\ell_1}, \ldots, v_{\ell_n})$. 
We let $b_1^\ell, \ldots, b_n^\ell, b^\ell\in \RR$ be such that
$$\nu=\sum_{i=1}^n b_i^\ell \nu_{\ell_i}+b^\ell \eta_\ell.$$
Equivalently,
$$v=\sum_{i=1}^\ell b_i^\ell v_{\ell_i}+b^\ell h_\ell\textup{ and }\sum_{i=1}^n b_i^\ell+\frac{k}{m}b^\ell=1.$$
By appropriately choosing $\eta_\ell$ we may assume $b^\ell>0$. For $t\in \ZZ^+$ we will denote
$$\iota_t=\#\left(\inte D\cap \frac{1}{t}\ZZ^n\right)=L_{\inte D}(t)=(-1)^n L_D(-t).$$

We will compute $\iota_t$ using the contact Betti numbers of $(D, \nu)$. Note that because the coordinates of $v$ are 
$\QQ$-independent there are no points in $\inte D\cap \frac{1}{t}\ZZ^n$ on $\partial D_\ell$, hence
$$\iota_t=\sum_{\ell} \#\left(\inte D_\ell \cap \frac{1}{t}\ZZ^n\right).$$

Take a point $p\in \inte D_\ell$. Such a point can be written uniquely as
\begin{align*}p&=\sum_{i=1}^n \alpha_i v_{\ell_i}+\alpha v
\end{align*}
with $\alpha_i, \alpha>0$ and $\sum_{i=1}^n \alpha_i+\alpha=1$. Then
\begin{align*}
(mp, m)&=\sum_{i=1}^n \alpha_i \nu_{\ell_i}+\alpha \nu\\
&=\sum_{i=1}^n (\alpha_i+\alpha b_i^\ell) \nu_{\ell_i}+\alpha b^\ell \eta^\ell.
\end{align*}
Hence, since $\nu_{\ell_1}, \ldots, \nu_{\ell_n}, \eta$ is a $\ZZ$-basis, $tp\in \ZZ^n$ if and only if $\alpha_i, \alpha$ 
are such that
$$t \alpha b^\ell\in m\ZZ \textup{ and }t(\alpha_i+\alpha b_i^\ell)\in m\ZZ \textup{ for }i=1, \ldots, n.$$

We now fix $N\in \ZZ^+$ and count the number of points $p$ in $\inte D_\ell\cap \frac{1}{t}\ZZ^n$ with $t\alpha b^\ell=mN$, 
that is, with $\alpha=\frac{mN}{tb^\ell}$. Then $t(\alpha_i+\alpha b_i^\ell)\in m\ZZ$ if and only if there is $m_i\in \ZZ$ such that 
$$\alpha_i=\frac{m}{t}\left(1-\left\{\frac{t\alpha b_i^\ell}{m}\right\}+m_i\right)=\frac{m}{t}\left(1-\left\{\frac{N b_i^\ell}{b^\ell}\right\}+m_i\right).$$
Moreover $\alpha_i>0$ if and only if $m_i\geq 0$.
Now $\sum_{i=1}^n \alpha_i^\ell+\alpha=1$ if and only if

\begin{align*}
1&=\frac{m}{t}\left( \sum_{i=1}^n\left(1-\left\{\frac{N b_i^\ell}{b^\ell}\right\}+m_i\right)+\frac{N}{b^\ell}\right)\\
&=\frac{m}{t}\left(\frac{1}{2} \deg \gamma_\ell^N +1+\sum_{i=1}^n m_i\right)
\end{align*}
which is equivalent to
$$\sum_{i=1}^n m_i=\frac{t}{m}-\frac{1}{2}\deg \gamma_\ell^N -1.$$
Above we used formula \eqref{eq:CZindexQ} to express the degree of $\gamma_\ell^N$ as
$$\frac{1}{2}\deg \gamma_\ell^N =\frac{N}{b}-\sum_{i=1}^n \left\{N\frac{b_i^\ell}{b^\ell}\right\}+n-1.$$

The number of non-negative integer solutions to such equations is found with the following well-known combinatorial lemma:

\begin{lemma}
Given $n, S\in \ZZ^+$ the number of solutions of 
$$\sum_{i=1}^n m_i=S \textup{ with }m_i\in \ZZ^+_0$$
is given by $\binom{S+n-1}{n-1}$.
\end{lemma}
\begin{proof}
The map from the family of such solutions to the family of subsets of $\{1, \ldots, S+n-1\}$ with $n-1$ elements that sends a 
solution $(m_1, \ldots, m_n)$ to the set 
$$\left\{\sum_{j=1}^k m_j+k: k=1, 2, \ldots, n-1\right\}\subseteq \{1, \ldots, S+n-1\}.$$
is easily seen to be a well defined bijection.\qedhere
\end{proof}

Hence, the number of points in $\inte D_\ell\cap \frac{1}{t}\ZZ^n$ with fixed $N$ is $\binom{\frac{t}{m}-\frac{1}{2}\deg \gamma_\ell^N+n-2}{n-1}$ 
where we interpret the binomial coefficient to be zero if $\frac{t}{m}-\frac{1}{2}\deg \gamma_\ell^N$ is not an integer (we remind the reader that 
the degree is not necessarily an integer). Therefore
\begin{align*}L_t(\inte D_\ell)&=\sum_{N\geq 1}\binom{\frac{t}{m}-\frac{1}{2}\deg \gamma_\ell^N+n-2}{n-1}\\
&=\sum_{j\in \frac{1}{m}\ZZ}\left(\#\left\{N\geq 1: \deg \gamma_\ell^N=2j\right\}\right)\binom{\frac{t}{m}-j+n-2}{n-1}.
\end{align*}
Summing over every facet $\ell$ it follows that 
$$\iota_t=\sum_{j\in \frac{1}{m}\ZZ}\binom{\frac{t}{m}-j+n-2}{n-1}cb_{2j}(D, \nu).$$
We now use  generating functions to recover $cb_{2j}= cb_{2j}(D, \nu)$. We compute the Ehrhart series of $\inte \Delta$ in terms of the 
contact Betti numbers:
\begin{align*}\Ehr_{\inte \Delta}(z)&=\sum_{t\geq 0}\iota_{t+m}z^t=\sum_{t\geq 0}\left(\sum_{j\in \frac{1}{m}\ZZ}\binom{\frac{t}{m}-j+n-1}{n-1}cb_{2j}\right)z^t.
\end{align*}
For the binomial coefficient to be non-zero we must have $t=mj+mi$ with $i\in \ZZ_{\geq 0}$; using such substitution, the Ehrhart series 
becomes
\begin{align}\Ehr_{\inte \Delta}(z)&=\sum_{j\in \frac{1}{m} \ZZ}cb_{2j}z^{mj}\sum_{i\geq 0}\binom{i+n-1}{n-1}z^{mi}=\frac{1}{(1-z^m)^n}\left(\sum_{j\in \frac{1}{m} \ZZ}cb_{2j}z^{mj} \right).\label{eq: ehrfromCH}
\end{align}
Note that we used in the last step the identity $\sum_{i\geq 0}\binom{i+n-1}{n-1} z^{mi}=\frac{1}{(1-z^m)^n}$.

The result follows from comparing \eqref{ehrinterior} and \eqref{eq: ehrfromCH}. More precisely, we evaluate the coefficient\footnote{We denote by $[z^j]F(z)$ the degree $j$ coefficient of a polynomial $F(z)$.}
$$[z^{mj}](1-z^m)^{n+1}\Ehr_{\inte \Delta}(z)\,,$$ 
in two different ways. Using equation \eqref{ehrinterior} this coefficient is equal to $\delta_{mn-mj}$. On the other hand, using 
\eqref{eq: ehrfromCH} it is equal to
$$[z^{mj}](1-z^m)\left(\sum_{j\in \frac{1}{m} \ZZ}cb_{2j}z^{mj} \right)=cb_{2j}-cb_{2(j-1)}.\qedhere $$
\end{proof}

As corollary we get that the contact Betti numbers stabilize at the normalized volume of~$D$.

\begin{theorem} \label{thm:combinatorial}
Let $D\subseteq \RR^n$ be a toric diagram and let $(M, \xi)$ be the $m$-Gorenstein toric contact manifold associated to $D$. Given 
$j\in \frac{1}{m}\ZZ \cap \left]-1,0\right]$, the sequence
$$\{cb_{2j+2a}(D)\}_{a\in \ZZ}$$
is monotonically increasing and stabilizes for large $a$. More precisely,
$$cb_{2j+2a}(D)=n!\vol(mD)\textup{ for }a\geq n.$$
\label{stabilization}
\end{theorem}
\begin{proof}
It follows from theorem \ref{main} that
$$cb_{2j+2a}(D)=\sum_{i=0}^a \delta_{m(n-j-i)}$$
so monotonicity follows immediately. For $a\geq n$,
$$cb_{2j+2a}(D)=\sum_{i=0}^n \delta_{m(n-j-i)}.$$
By \eqref{Ehrhartbasis}, $\frac{1}{n!}\left(\sum_{i=0}^n \delta_{m(n-j-i)}\right)$ is the leading coefficient of the polynomial 
$L_\Delta(mt-mj)$ which is equal to $\vol(mD)$.
\end{proof}

\begin{remark}
The stabilization result implies that the mean Euler characteristic of $(M, \xi)$, defined in this case by
$$\chi(D, \nu)=\lim_{N\to +\infty} \frac{1}{2N} \sum_{j\in [0, N]} cb_{2j}(D, \nu),$$
is given by $\frac{n!}{2}m\textup{vol}(mD)$. The Gorenstein case $m=1$ was proven by the first two authors in~\cite{abreu2016mean}. 
It was also shown that the mean Euler characteristic was the orbifold Euler characteristic of any crepant toric symplectic filling 
(using \cite{batyrev1996strong}); this fact also follows from Theorem \ref{stapledon}.
\end{remark}

Theorem \ref{main} can be used to give direct combinatorial interpretations for the dimensions of other contact homology groups.

\begin{remark} \label{rem:combinatorial}
\leavevmode
\begin{enumerate}
\item Since $\delta_0=1$ it also follows that $$cb_{2(n-1)}=n!\vol(mD)-1.$$
\item We have
$$cb_0(D)=\delta_{mn}=L_{\inte D}(m)=\#\left(\inte(mD)\cap \ZZ^n\right),$$
where the second equality follows by making $z=0$ in~(\ref{ehrinterior}).
\end{enumerate}
\end{remark}

\section{Resolutions of the symplectic cone}
\label{resolutions}

Recall from section \ref{toriccontact} that associated to a toric diagram $D$ we have a toric symplectic cone $(W, \omega, X, \mu)$. Its moment cone is $C=\mu(W)\cup \{0\}$. The symplectic cone $W$ can be compactified near $\mu^{-1}(B_\epsilon(0))$ by adding a single point corresponding to the vertex $0$ of the moment cone $C$; this creates a toric singularity that can be described in algebraic geometric terms as a toric variety. 

Given the toric diagram $D\subseteq \RR^n$, its cone $\sigma$ is the cone over $D\times \{1\}$ in $\RR^{n+1}$, that is, 
$$\sigma=\{(tx, t): x\in D, t\geq 0\}\subseteq \RR^{n+1}.$$
Note that this is the cone generated by $\nu_i=(mv_i, m)$ where $\nu_i$ are the vectors normal to the facets of $C$ and $v_i$ are the vertices of $D$, hence $\sigma$ is the dual of the cone $C$. Associated to $\sigma$ we have an affine toric variety $\overline W=X_\sigma$ which is the union of $W$ with the singularity (see \cite{cox}). 

Now a subdivision $\T$ of $D$ induces a fan $\Sigma$ refining the (fan given by the) cone $\sigma$, and hence gives a toric variety $X_\Sigma$ and a (partial) resolution of the singularity $X_\Sigma\to X_\sigma=\overline W$; we call this the fan over $\T$. The fan $\Sigma$ consists in the family of cones over $\theta\times \{1\}$ where $\theta\in \T$; in particular there is a correspondence between $\T_d$ and $\Sigma(d+1)$. When $\T$ is a rational triangulation (i.e., $m\T\subseteq \ZZ^n$ where $m$ is the order of the toric diagram $D$) such a resolution is a toric crepant resolution since every generator of cones in $\Sigma$ has last coordinate $m$ (cf. \cite[Proposition 11.2.8.]{cox}). 

The resolution $X_\Sigma$ is smooth if and only if the minimal generators of every cone of $\Sigma$ can be extended to a basis of $\ZZ^{n+1}$, which is equivalent to $m=1$ and the subdivision $\T$ being a unimodular triangulation. If $m>1$ or $m=1$ and $\T$ is a non-unimodular triangulation then $X_\Sigma$ is an orbifold. 
Note that every polytope admits a triangulation and when $n\leq 2$ it always admits a unimodular triangulation, but for $n>2$ that is no longer true; in particular not every such cone $\overline W$ admits a smooth crepant toric resolution.

\begin{remark}
In \cite{mclean2016reeb} McLean proved a result that implies that the minimal discrepancy of the isolated singularity in $X_\sigma$ is half the degree of the first non-trivial contact homology group. By Theorem \ref{main} this is the same as the smallest $r$ for which there is an integral point in $(r+1)\textup{int}\, (mD)$. In particular the singularity is terminal if and only if there are no integral points in the interior of $mD$. Thus, if the singularity is not terminal, it admits a (partial) toric resolution given by the star subdivision centered at some integral point.
\end{remark}

\subsection{Symplectic structure on $X_\Sigma$}
\label{Xsigma}

We explain here how we can give symplectic structures to $X_\Sigma$, which a priori is an abstract algebraic variety. When we give a (toric) symplectic structure to $X_\Sigma$ we should get a convex polyhedral set $P$ with normal fan $\Sigma$; that is $P$ should take the form
$$P=\{x\in \RR^{n+1}: \langle x, (v,1) \rangle\geq a_v \textup{ for all }v\in \T_0\}$$
for some constants $a_v\in \RR$. Note that for $\Sigma$ being the normal fan of $P$ the constants $a_v$ have to be chosen in a way that the faces of $P$ are dual to the cones of $\Sigma$. This condition translates as follows: given $T\in \T_k$, 
$$P\cap \{ \langle x, (v,1) \rangle=a_v \textup{ for all }v\in T\cap \T_0\}$$
is a face of $P$ of codimension $k+1$. To state the conditions in which this happens we introduce the following notions:

\begin{definition} \label{def:support}
 Let $\Sigma\subseteq \RR^{n+1}$ be a fan. A support function $\varphi$ on $\Sigma$ is a function $\varphi: |\Sigma|\to \RR$ such that $\varphi$ is linear in each cone of $\Sigma$. We say that $\varphi$ is convex if 
 $$\varphi(\lambda x+(1-\lambda)y)\leq \lambda\varphi(x)+(1-\lambda)\varphi(y) \, \, \, \forall x,y\in |\Sigma|.$$
 We say $\varphi$ is strictly convex if it is convex and the above inequality is strict for every $x, y\in |\Sigma|$ 
 such that there is no cone $\sigma\in \Sigma$ containing $x$ and $y$.
 \end{definition}
 
 Note that a support function is uniquely determined by choosing its values along rays $\rho\in \Sigma(1)$, where 
 $\rho=\RR_{\geq 0}\cdot (v,1)$ and $v\in \T_0$, or alternatively, by choosing the values $a_v=\varphi(v,1)$. 
 Alternatively, $\varphi$ is determined by its Cartier data $\{m_T\}_{T\in \T_n}$ where $m_T\in \RR^{n+1}$ is defined by 
 asking that $\varphi(x)=\langle m_T, x\rangle$ when $x\in \sigma$ and $\sigma\in \Sigma(n+1)$ is the cone over $T$. 
 Given a support function $\varphi$, we define its associated polyhedral set $P$ as
\begin{align*}P&=\{x\in \RR^{n+1}: \langle x, (v,1) \rangle\geq \varphi(v,1) \textup{ for all }v\in \T_0\}\\
&=\{x\in \RR^{n+1}: \langle x, \nu \rangle\geq \varphi(\nu) \textup{ for all }\nu \in |\Sigma|\}. \end{align*}

\begin{proposition}[Lemma 6.13, \cite{cox}]
Given a fan $\Sigma$ and a support function $\varphi$ on $\Sigma$, the polyhedral set above defined has normal fan $\Sigma$ if and only if $\varphi$ is strictly convex. Moreover, in this case, the vertices of $P$ are precisely given by the Cartier data $\{m_T\}_{T\in \T_n}$.
\label{normalfan}
\end{proposition}

Now from $P$ we can construct a toric symplectic manifold having $P$ as the image of its moment map. This can be done using the symplectic cutting construction as presented in \cite{okitsu2013cutting} (and based on \cite{lerman1995symplectic}). Moreover this manifold has the structure of an algebraic variety with fan $\Sigma$. We state here this result and we refer to the literature for its proof, mentioning some adjustments.

\begin{proposition} \label{prop:symplectic}
Let $\Sigma\subseteq \RR^{n+1}$ be a simplicial, full dimensional fan admitting a strictly convex support function $\varphi$. Then $X_\Sigma$ admits a symplectic structure $\omega_\varphi$ making it a toric symplectic orbifold with moment map image 
$$P=\{x\in \RR^{n+1}: \langle x, \nu \rangle\geq \varphi(\nu) \textup{ for all }\nu \in |\Sigma|\}.$$
\end{proposition}
\begin{proof}
Since $\Sigma$ is simplicial and full dimensional, $P$ is a simple and strongly convex polyhedral set. In \cite{okitsu2013cutting} a toric symplectic orbifold with moment map image $P$ is constructed by starting with $T^\ast \mathbb{T}^{n+1}$ and performing successive symplectic cuts (introduced in \cite{lerman1995symplectic}). Note that although the construction is stated only for unimodular convex polyhedral set it works when we drop this condition as long as it is still simple. Indeed, in the intermediate steps of the cutting process we get orbifolds, and the unimodularity condition is only used in Remark 2.8 to ensure that the final manifold is smooth.

Theorem 5.1 in \cite{okitsu2013cutting} gives also a compatible Kähler structure which is $\mathbb T^{n+1}$-invariant. 
Now Lemma 9.2 in \cite{lerman1997hamiltonian} shows that this Kähler orbifold can be given the structure of a toric variety 
isomorphic to $X_\Sigma$, since by Proposition \ref{normalfan} the normal fan of $P$ is $\Sigma$. Note that the result there 
is only stated in the case that the orbifold is compact (equivalently, $\Sigma$ is complete, or $P$ is bounded) but the proof 
works verbatim to the non-compact case. We remark that in this proof it is implicitly shown that given $\sigma\in \Sigma(n+1)$, 
with corresponding vertex $m_\sigma\in P$, we have that
$$\mu^{-1}\left(\bigcup_{m_\sigma\in F} \textup{int }F\right),$$
has the structure of an affine toric variety with cone $\sigma$, where the union runs over faces of $P$ containing $m_\sigma$. 
\qedhere
\end{proof}

Note that although the underlying algebraic variety structure $X_\Sigma$ is always the same, the symplectic form depends a 
lot on $\varphi$. The polyhedral set $P$ always has the same combinatorial structure, but by changing $\varphi$ we change 
for instance the (lattice) length of its edges which are the symplectic areas of the corresponding spheres.

\begin{remark}
When $\Sigma$ is complete the results of chapter 6 of \cite{cox} show that from a strictly convex support function $\varphi$ 
we get a (toric-invariant) ample divisor and thus an embedding $X_\Sigma \hookrightarrow \CC P^{s}$, for some $s\in\NN$. 
In this case the symplectic form $\omega_\varphi$ can be obtained by pulling-back the Fubini-Study form in $\CC P^s$. 
\end{remark}

This shows that given a toric $\QQ$-Gorenstein contact manifold $M$ with toric diagram $D$, a triangulation $\mathcal T$ of 
$D$ and a strictly convex support function defined on the fan $\Sigma$ over $D$ one gets a crepant symplectic filling 
$(X_\Sigma, \omega_\varphi)$ of $M$, which is smooth if and only if $\T$ is unimodular -- we call a filling obtained in this way a 
crepant toric filling. We now prove that, in this context, strictly convex support functions always exist.

\begin{proposition} \label{prop:filling}
Let $D$ be a toric diagram, $\T$ a triangulation and $\Sigma$ the fan over $\T$. Then $\Sigma$ admits a strictly convex 
support function. In particular every compact $\QQ$-Gorenstein contact toric manifold admits a not-necessarily smooth 
$\QQ$-crepant symplectic filling.
\end{proposition}

\begin{proof}
We note that the proof of Theorem 6.1.18 in \cite{cox} adapts to this case (although the result proved there does not apply directly). 
For each $\tau\in \Sigma(n)$ let $m_\tau\in \RR^{n+1}$ be any vector such that $\tau=\{x\in \RR^{n+1}: \langle x, m_\tau \rangle=0\}$ 
(this is unique up to scaling) and define a support function $\varphi: |\Sigma|\to \RR$ by
$$\varphi(x)=\sum_{\tau\in \Sigma(n)}|\langle x, m_\tau \rangle|.$$
Convexity follows from triangle inequality and we get strict convexity by noticing that if $x, y$ are in two different cones 
$\sigma, \sigma'\in \Sigma(n+1)$ then there is some $\tau$ for which the hyperplane $\textup{span}(\tau)$ separates $x$ and $y$, 
hence $\langle x, m_\tau\rangle$ and $\langle y, m_\tau\rangle$ have different signs and we cannot have equality in the triangle 
inequality.

It remains to show that $\varphi$ is piecewise linear in every cone $\sigma\in \Sigma(n+1)$. By the construction of $\Sigma$ as 
a fan over $\T$, the hyperplane $\textup{span}(\tau)$ does not intersect the interior of $\sigma$ (this is why we do not have to take 
a refinement of $\Sigma$ as in \cite{cox}). Therefore $\langle x, m_\tau\rangle$ does not change sign for $x\in \sigma$ and hence 
$\varphi_{|\sigma}$ is linear. \qedhere
\end{proof}

\begin{example}
Consider the toric diagram $D=\textup{conv}((0,0), (1,0), (0,1), (2,2))$. This gives a Gorenstein contact manifold diffeomorphic to $S^2\times S^3$ (see \cite[Section 6.1]{abreu2012contact}). We have two combinatorially distinct unimodular triangulations of $D$ that give two different smooth crepant toric fillings:
\begin{center}
\includegraphics[scale=0.4]{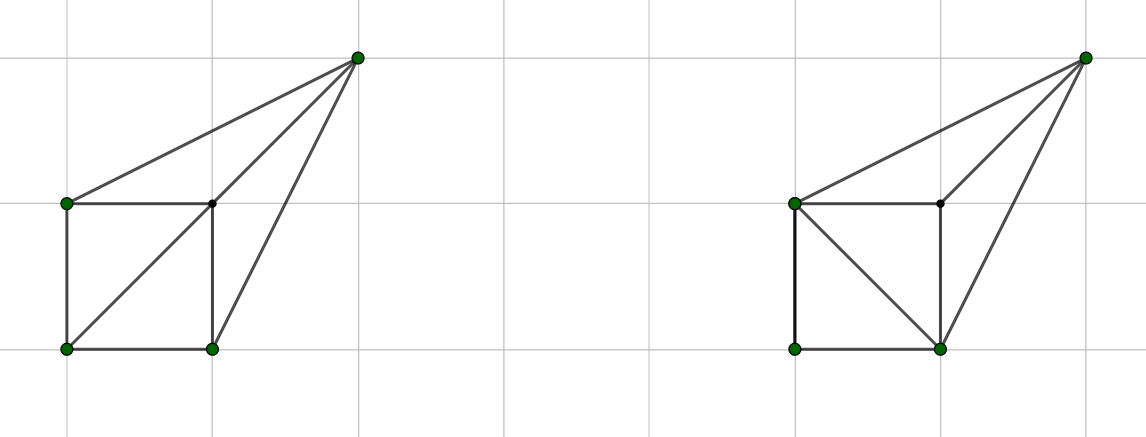}
\end{center}
The first resolution is the total space of the canonical bundle $K_{\mathbb F_1}$ where $\mathbb F_1$ is the Hirzebruch surface $\mathbb P(\Oo_{\mathbb P^1}\oplus \Oo_{\mathbb P^1}(1))$. The two resolutions are related by an Atiyah flop.
\end{example}

\subsection{Orbifold cohomology of $X_\Sigma$}

A result of Stapledon in \cite{stapledon} relates the Ehrhart polynomial of $D$ with the orbifold cohomology of the varieties $X_\Sigma$ described above. 

The orbifold cohomology is an orbifold invariant introduced by Chen and Ruan in \cite{chen2004new}. In general it assigns a $\QQ$-graded ring $H^\ast_\textup{orb}(X)$ to an orbifold $X$. This grading is actually a $\frac{1}{m}\ZZ$-grading when the orbifold is $\QQ$-Gorenstein of order $m$. If the orbifold is smooth, then the orbifold cohomology ring is isomorphic to the singular cohomology.

\begin{theorem}
Let $D$ be a rational toric diagram of order $m$ and $\T$ a rational triangulation of $D$ (i.e., $m\T_0\subseteq \ZZ^n$).  Let $\Sigma$ be the fan over $\T$ and $X_\Sigma$ the toric variety with fan $\Sigma$. Then 
$$\dim H^{2j}_{\textup{orb}}(X_\Sigma; \QQ)=\delta_{mj}$$
for $j\in \frac{1}{m}\ZZ$, and the remaining orbifold cohomology groups are trivial. 
\label{stapledon}
\end{theorem}

This fact was known when $X_\Sigma$ is smooth at least since \cite{batyrev1996strong}. The Gorenstein case ($m=1$) was proven in \cite[Theorem 4.6]{stapledon}. We give an adaptation of Stapledon's argument for the rational case in the appendix, see Theorem \ref{stapledongeneral}.

Combining this result with Theorem \ref{main} we get the contact Betti numbers of $D$ from the orbifold cohomology of $X_\Sigma$.

\begin{corollary}
Let $D\subseteq \RR^n$ be a rational toric diagram and $(M, \xi)$ its associated $\QQ$-Gorenstein toric contact manifold. Let $\T$ be a rational triangulation of $D$, $\Sigma$ the fan over $\T$ and $X_\Sigma$ the toric variety with fan $\Sigma$. Then we have for every $j\in \QQ$
$$cb_{2j}(D)=\sum_{k\geq 0} \dim H_{\textup{orb}}^{2n-2j+2k}(X_\Sigma; \QQ).$$
\label{isostapledon}
\end{corollary}

\begin{proof}
By Theorems \ref{main} (for the left hand side) and \ref{stapledon} (for the right hand side) both sides are equal to $\sum_{0\leq k<j+1}\delta_{m(n-j+k)}$.  \qedhere
\end{proof}

\subsection{Smooth case of Corollary \ref{isostapledon}} \label{ss:isostapledon}

We give here a direct proof of Corollary \ref{isostapledon} when $m=1$ and $X_\Sigma$ is smooth (equivalently, $D$ is an integral toric diagram and $\T$ is a unimodular triangulation) that does not go through the Ehrhart polynomial. Recall that if $X_\Sigma$ is smooth then the orbifold cohomology is just the singular cohomology. 

The proof uses symplectic (co)homology of the contact type boundary $X_\Sigma$\footnote{To be precise, $X_\Sigma$ is the symplectic completion of a contact type boundary symplectic manifold.} and its positive/negative and $S^1$-equivariant versions. We refer to \cite{viterbo1999functors,oancea2013gysin,oancea2014linearized} for details on the constructions and for the results we will need. We will mostly follow the grading conventions of \cite{oancea2014linearized} with two differences: the dimension of our symplectic completion $X_\Sigma$ is $2n+2$ instead of $2n$, and we use the symplectic field theory shift of $n-2$ in the index of contact homology.

The key fact that will be necessary in the proof is the vanishing of symplectic (co)homology:

\begin{lemma}
Let $D$ be a toric diagram, $\T$ a unimodular triangulation and $\Sigma$ the fan over $\T$. Let $X_\Sigma$ be the smooth toric symplectic manifold described in \ref{Xsigma}. Then $SH_\ast(X_\Sigma)=0$ and $SH_\ast^{S^1}(X_\Sigma)=0$.
\end{lemma}
\begin{proof}
We let $\nu=(v, 1)\in \RR^{n+1}$ be a vector associated to a certain Reeb vector field $R_\nu$ in $M$. 
Given $\alpha\in \RR^+$ we define a Hamiltonian $H_\alpha: X_\Sigma\to \RR$ by 
$$H_\alpha(x)=\alpha\langle \mu(x), \nu\rangle$$
where $\mu: X_\Sigma\to \RR^{n+1}$ is the moment map of $X_\Sigma$ (whose image is a cone $P$). 
The corresponding Hamiltonian vector field is toric and so, for each $\alpha\in \RR^+$, there are only finitely many 
periods of its simple non-constant closed orbits. Hence, for almost all values of $\alpha\in \RR^+$ the only $1$-periodic 
Hamiltonian orbits of $H_\alpha$ are the constant orbits corresponding to critical points of $H_\alpha$ and we will only consider
those values of $\alpha$. 
The critical points of $H_\alpha$ correspond to vertices of $P$. More precisely, if $m$ is a vertex of $P$ 
then $\mu^{-1}(m)=\{p\}$ and $p$ is a critical point of $H_\alpha$. 
We will show that the indices of all the constant orbits $c_p$ get arbitrarily large when we let $\alpha$ go to 
$+\infty$, which proves the vanishing claimed by the definition of symplectic homology as a colimit of Floer homology of Hamiltonians with slope going to $+\infty$.

Let $m\in P$ be a vertex and let $\nu_1, \ldots, \nu_{n+1}$ be the normals to the facets intersecting at $m$. Write 
$\nu=\sum_{j=1}^{n+1} b_j \nu_j$; note that this equality implies $\sum_{j=1}^{n+1}b_j=1$ by looking at the last coordinate. By changing coordinates we may assume that $\nu_j=e_j$ are the coordinate vectors and $\nu=(b_1, \ldots, b_{n+1})$, thus, near $p$ we can give complex coordinates $z_1, \ldots, z_{n+1}$ to $X_\Sigma$ such that $p$ corresponds to all $z_j=0$ and 
$$H_\alpha(z_1, \ldots, z_{n+1})=\sum_{j=1}^{n+1}\frac{\alpha b_j}{2}|z_j|^2.$$
Thus the flow (near $p$) is given by
$$\varphi_t(z_1, \ldots, z_{n+1})=\left(e^{2\pi it\alpha b_1}z_1, \ldots, e^{2\pi it\alpha b_d}z_1\right).$$
The above condition on the values of $\alpha$ implies that $\alpha b_j \not\in \ZZ$, $j=1,\ldots, n+1$,
and the Conley-Zehnder index is
$$\sum_{j=1}^{n+1}\left(2\lfloor \alpha b_j\rfloor+1\right)>\sum_{j=1}^{n+1}\left(2\alpha b_j-1\right)=2\alpha-n-1.\qedhere$$

\end{proof}

The long exact sequence for positive/negative symplectic (co)homology (see \cite[Proposition 1.5]{viterbo1999functors} and \cite[Lemma 4.8]{oancea2013gysin}) gives isomorphisms
$$SH_\ast^{+}(X_\Sigma)\overset{\cong}{\rightarrow} H_{\ast+n}(X_\Sigma, M)\textup{ and }SH_\ast^{+, S^1}(X_\Sigma)\overset{\cong}{\rightarrow} H_{\ast+n}^{S^1}(X_\Sigma, M).$$
Here $H_{\ast+n}^{S^1}(X_\Sigma, M)$ denotes $S^1$-equivariant homology with respect to the trivial $S^1$-action on $(X_\Sigma, M)$, that is, $H_{\ast}^{S^1}(X_\Sigma, M)\cong H_\ast(X_\Sigma, M)\otimes \ZZ[u]$ where $u$ is a generator in degree $2$. Thus as vector spaces we have
$$SH^{+, S^1}_\ast(X_\Sigma)=\bigoplus_{k\geq 0}H_{\ast+n-2k}(X_\Sigma, M)\cong \bigoplus_{k\geq 0}H^{n+2-\ast+2k}(X_\Sigma)$$
where we used Lefschetz duality for the last isomorphism. The isomorphism between (linearized) contact homology and positive 
equivariant symplectic homology in \cite[Theorem 1.4]{oancea2014linearized} 
then gives\footnote{The shift in the index between contact homology and positive equivariant symplectic homology is the SFT 
shift in contact homology, which was not used in \cite{oancea2014linearized}.}
$$HC_\ast(M)=SH^{+, S^1}_{\ast-n+2}(X_\Sigma)\cong \bigoplus_{k\geq 0}H^{2n-\ast+2k}(X_\Sigma).$$

\section{Contact homology of prequantization}
\label{baseorbifold}
An important way in which contact manifolds appear is as prequantizations of symplectic manifolds/orbifolds; we briefly explain this construction in the toric case. Let $B$ be a compact symplectic toric orbifold of dimension $2n$. Compact symplectic toric orbifolds are classified by simple rational polytopes with facets labelled by positive integers (see \cite{lerman2002contact} or \cite{lerman1997hamiltonian}).
The polytope $\Delta$ is the image of the moment map on $B$ and the label of a facet $F$ is the order of the structure group of points in $\mu^{-1}(\inte F)$. We write the polytope $\Delta$ as
\begin{equation}\Delta=\{x\in \RR^{n}: \langle x, v_i \rangle+b_i\leq 0, \textup{ for } i=1, \ldots, d\}.\label{polytope}\end{equation}
It is convenient to choose $v_i$ to be given by $v_i=\eta_i \widetilde{v_i}$ where $\eta_i$ is the label on the facet normal to $v_i$ and $\widetilde{v_i}\in \ZZ^n$ is the primitive inwards pointing vector; we will call these $v_i$ the weighted normals. From the weighted normals we can recover the labels $\eta_i$ as the greatest common divisor of the components of $v_i$. The construction of $B$ from the labelled polytope is similar to the Delzant construction with weighted normals $v_i$ replacing the primitive normals. 

We now define prequantization of orbifolds. This construction in the smooth case was introduced by Boothby and Wang in \cite{boothby1958contact} and was adapted to the orbifold case in \cite{thomas1976almost}; contact forms arising from prequantization of manifolds (orbifolds) are called (almost) regular contact forms.

\begin{definition}
Let $(M, \xi)$ be a contact $(2n+1)$-manifold with contact form $\alpha$ and let $(B, \omega)$ be a symplectic $2n$-orbifold. We say that $M$ is the prequantization of $B$ if 
\begin{enumerate}
\item The flow of the Reeb vector field $R_\alpha$ is $1$-periodic and induces an almost free action of $S^1=\RR/\ZZ$ on $M$;
\item $B$ is the quotient space $M/S^1$; so there is a principal $S^1$-orbibundle $\pi: M\to B$;
\item $\pi^\ast \omega=d\alpha$.
\label{prequantization}
\end{enumerate}
\end{definition}

Note that $\alpha\in \Omega^1(M)$ is the connection form, $d\alpha\in \Omega^2(M)$ the curvature form and by (3) in the definition $[\omega]\in H^2_{dR}(B)$ is the characteristic class classifying the $S^1$-orbibundle $M\to B$. In particular $[\omega]$ has to be integral (for details on de Rham cohomology of orbifolds see \cite[Section 2.1]{adem2007orbifolds}). Conversely, from a symplectic orbifold $(B, \omega)$ such that $[\omega]$ is an integral cohomology class one can construct a $S^1$-orbibundle with a contact form $\alpha$.

Another way to describe prequantization is as follows: we take $W'\to B$ to be the complex line bundle with Chern class $[\omega]$ and let $M\subseteq W'$ be the unit circle bundle (giving some Hermitian structure to $M\to W'$); then $W=W'\setminus\{\textup{zero section}\}$ is the symplectization of $M$. Moreover, the symplectic reduction of $W$ with respect to the $S^1$ action gives back $B$. 

\subsection{Prequantization of toric orbifolds}

Let $(B, \omega)$ be a compact symplectic toric orbifold as before, classified by a moment polytope $\Delta=\mu(B)$ as in \eqref{polytope} and set of weights $\{\eta_i\}$. According to theorem 6.3 in \cite{guillemin1994kaehler} the Poincaré dual of the class of $\omega$ is
$$PD[\omega]=\sum_{i=1}^d b_i[\mu^{-1}(F_i)]$$
where $[\mu^{-1}(F_i)]\in H_{2n-2}(B)$ is the homology class represented by the manifold that is the pre-image of the facet $F_i$ (normal to $v_i$) under the moment map. To guarantee that this class is integral we assume that $b_i\in \ZZ$. Note that, unlike in the smooth/Delzant case, this does not imply that $\Delta$ is integral. 

In \cite{lerman2002maximal} the prequantization (in the smooth case) is described as follows: let $C\subseteq \RR^{n+1}$ be the cone over $\Delta\times \{1\}$, that is, 
\begin{align}
C&=\{t(x,1): t\in \RR^+_0, x\in \Delta\}\nonumber \\
&=\{y\in \RR^{n+1}: \langle y, \nu_i\rangle \geq 0 \textup{ for }i=1, \ldots, d\}
\label{cone}
\end{align}
where $\nu_i=(v_i, b_i)\in \RR^{n+1}$. Then the prequantization of $B$ is the contact manifold $M$  corresponding to $C$ described in section \ref{toriccontact} with contact form $\alpha=\alpha_{e_{n+1}}$ corresponding to the toric Reeb vector $e_{n+1}$. This is still true if $B$ is a orbifold.

\begin{proposition}[Lemma 3.7, \cite{lerman2002maximal}]
Let $(B, \omega)$ be a $2n$-dimensional compact symplectic toric orbifold with labelled moment polytope $(\Delta, \{\eta_i\})$ as in \eqref{polytope}. Let $C\subseteq \RR^{n+1}$ be the cone in \eqref{cone}; assume that $C$ is smooth. Now let $(M, \alpha)$ be the contact toric manifold corresponding to $C$ where $\alpha$ is the toric contact form with $R_\alpha=R_{e_{n+1}}$. Then the prequantization of $(B, \omega)$ is $(M, \alpha)$.
\end{proposition}

\begin{proof}
This was proved when $B$ is smooth in \cite{lerman2002maximal}. The proof works just the same except that instead of proving 
that the $S^1$ action induced by the Reeb flow of $R_\alpha=R_{e_{n+1}}$ is free we prove that it is almost free. It was argued that 
to show the freeness of the action it was enough to check in the rays of $\RR^+(v^\ast, 1)$ of $C$ where $v^\ast$ is a vertex of 
$\Delta$ and that this follows from the fact that $\{(v_\ell, b_\ell)\}_\ell \cup \{e_{n+1}\}\subseteq \ZZ^{n+1}$ is a $\ZZ$-basis where 
$\ell$ runs the indices of the $n$ facets of $\Delta$ whose intersection is $v^\ast$; this followed by the Delzant condition on $\Delta$. 
This is not true anymore, but it is true that $\{(v_\ell, b_\ell)\}_\ell \cup \{e_{n+1}\}$ is linearly independent, so it generates a lattice with 
finite index over $\ZZ^{n+1}$, and this index is the order of the isotropy group of $x$ for $x$ with image in the ray $\RR^+(v^\ast, 1)$.
\qedhere
\end{proof}

\begin{remark}
We will always assume $\Delta$ to be such that the cone $C$ over $\Delta$ is good, as this is necessary for $M$ to be smooth (otherwise $M$ would be a contact orbifold).
\end{remark}

We now give a description of the polytopes $\Delta$ for which the prequantization $M$ is Gorenstein. Note that in this section we will only consider Gorenstein contact manifolds (meaning that $c_1(\xi)=0$ in $H^2(M;\ZZ)$ or, equivalently, $m=1$ in the notation of the previous sections) and not $\QQ$-Gorenstein as we did previously.

\begin{definition}
We say that a polytope $\Delta$ as in \eqref{polytope} is almost-reflexive if $b_i=1$ for $i=1, \ldots, d$.

We say that $\Delta$ is $r$-Gorenstein if $r\Delta$ is almost-reflexive (up to translation) and the cone $C$ over $\Delta$ is a good cone. 
\end{definition}

Note that an integral almost-reflexive polytope is a reflexive polytope.

\begin{proposition}
Let $(B, \omega)$ be a compact toric symplectic orbifold with labelled polytope $\Delta$. Then its prequantization $(M, \alpha)$ is a smooth toric Gorenstein contact manifold if and only if $\Delta$ is $r$-Gorenstein for some $r\in \ZZ^+$.\label{rgorenstein}
\end{proposition}

\begin{proof}
The condition that the cone $C$ is a smooth cone is necessary to ensure that $M$ is smooth. 

By Proposition~\ref{prop:c_1}, the prequantization $M$ of $B$ has trivial first Chern class if and only if there is $\nu^\ast\in (\ZZ^{n+1})^\ast$ such that $\nu^\ast(\nu_j)=1$ for $j=1, \ldots, d$, where $\nu_j=(v_j, b_j)$. Any such $\nu^\ast$ can be uniquely written as $\nu^\ast(y)=\langle (w, r), y\rangle$ for some $w\in \ZZ^n, r\in \ZZ$. So $M$ has trivial first Chern class if and only if there are $w, r$ such that
$$\langle x, \nu_j\rangle +b_j\geq 0\Leftrightarrow \langle rx, \nu_j\rangle +(1-\langle w, \nu_j\rangle)\geq 0 \Leftrightarrow \langle rx-w, \nu_j\rangle+1\leq 0.$$
But this condition is equivalent to $r\Delta$ being almost-reflexive (after a translation by $w$).\qedhere
\end{proof}

\begin{remark}\label{rem: deltaDcorrespondence}
Note that when $\Delta$ is $1-$Gorenstein, in which case we say that $\Delta$ is orbi-reflexive, the cone $C$ has normals $\nu_i=(v_i,1)$, so the toric diagram associated to $M$ is $D=\textup{conv}(v_1, \ldots, v_d)=\Delta^\circ$ where $v_i$ are the (weighted) normals of $\Delta$. When $r>1$ we have to take a change of basis to transform $\nu_j=(v_j, b_j)$ in $\mu_j=(u_j, 1)$ which amounts to finding a matrix $A\in M_{n, n+1}(\ZZ)$ such that 
$$\begin{bmatrix}
A\\ 
\begin{matrix}
-w-&
r
\end{matrix}
\end{bmatrix}\in GL_{n+1}(\ZZ).$$
In this case $u_j=A \nu_j^t$ and $D=\textup{conv}(u_1, \ldots,  u_d)$.
\end{remark}

We also remark that any toric (Gorenstein) contact manifold is the prequantization of some orbifold. Indeed suppose that $\frac{w}{r}\in \textup{int}\, D\cap \QQ^n$ with $w\in \ZZ^n$ and $r\in \ZZ^+$. Then the Reeb flow of the contact form $\alpha_\nu$ where $\nu=(w, r)$ induces an $S^1$ action on $M$ and the quotient $M/S^1$ is a symplectic toric orbifold $B$ with a moment polytope $\Delta$ which is $r$-Gorenstein and which can be obtained explicitly by reverting the above construction. In a basis free way, the normals to the facets of $\Delta$ are the projections of $\mu_j$ in $\ZZ^{n+1}/\nu \ZZ\cong \ZZ^n$. 

At least in the case that $B$ is smooth, this number $r$ can be characterized geometrically.

\begin{proposition}
\label{chernclassr}
Suppose we have the conditions of Proposition \ref{rgorenstein} and moreover $B$ is smooth. Then $B$ is monotone and $c_1(TB)=r[\omega].$
\end{proposition}
\begin{proof}
If $r=1$ then $\Delta$ is a reflexive polytope and this is well known. 
In the general case we note that scaling the polytope $\Delta$ by a factor of $r$ scales $[\omega]$ by a factor of $r$ and does not affect $c_1(TB)$.
\end{proof}

Assume $B$ is smooth. The long exact sequence for the $S^1$-bundle gives
$$\pi_2(M)\rightarrow \pi_2(B)\overset{\partial}{\rightarrow} \pi_1(S^1)\cong \ZZ\rightarrow \pi_1(M)\rightarrow \pi_1(B).$$

We claim that the map $\partial$ is given by evaluating $[\beta]\in \pi_2(B)$ at the characteristic class $[\omega]$. Let $\beta: S^2\to B$ be a representative of $[\beta]\in \pi_2(B)$ and let $\widetilde \beta: D^2\to M$ be the composition of $\beta$ with the map $D^2\to S^2$ that collapses the boundary $\partial D^2\cong S^1$ to a point. Then $[\partial \beta]=[\widetilde \beta_{|S^1}]$ and we have
$$\langle[\omega], [\beta]\rangle =\int_{S^2}\beta^\ast \omega=\int_{D^2}\widetilde \beta^\ast \pi^\ast \omega=\int_{D^2}\widetilde \beta^\ast d\alpha=\int_{S^1}
(\widetilde \beta_{|S^1})^\ast \alpha=\deg(\partial [\beta]).$$

It follows that, if $B$ is simply connected as is always the case for toric symplectic manifolds, then:
$$\pi_1(M)\cong \ZZ/\langle [\omega], \pi_2(B)\rangle.$$

\begin{remark}
Each edge $e\in \Delta_1$ corresponds to a sphere $S_e=\mu^{-1}(e)\subseteq B$. It is well known that the classes $[S_e]$ of these spheres generate $\pi_2(B)\cong H_2(B)$.

Moreover $\langle [\omega], [S_e]\rangle=\ell(e)$ is the integral length of $e$ (that is, the number of integral interior points +1), so it follows that $\langle [\omega], \pi_2(B)\rangle=p\ZZ$ where
$$p=\gcd\{\ell(e): e\in \Delta_1\}.$$
Comparing with Proposition 2.10 in~\cite{abreu2018contact}, it follows that
$$p=\gcd\left\{\begin{vmatrix}
\vert &  & \vert\\ 
\nu_{i_1} & \ldots & \nu_{i_{n+1}}\\ 
\vert &  & \vert
\end{vmatrix}: 1\leq i_1<\ldots<i_{n+1}\leq d\right\}.$$ 
It is possible to prove this combinatorial fact directly.
\end{remark}

This last fact gives a relation between $r$ and the minimal Chern number. Recall that the minimal Chern number of $B$ is $k\in \ZZ^+$ such that $\langle c_1(TB), \pi_2(B)\rangle=k\ZZ$.

\begin{corollary}
Suppose we have the conditions of Proposition \ref{rgorenstein} and moreover $B$ is smooth. Then the minimal Chern number of $B$ is $k=rp$ where $p=|\pi_1(M)|$.
\end{corollary}

\begin{proof}
This is clear from proposition \ref{chernclassr} and the isomorphism $$\pi_1(M)\cong \ZZ/\langle [\omega], \pi_2(B)\rangle.\qedhere$$
\end{proof}

\subsection{Contact homology of $M$ from $B$} 

We will now prove Theorem \ref{thm:quotients}, which relates the orbifold cohomology of the base $B$ and the contact homology of its prequantization $M$.  For the convenience of the reader we restate it now.

Let $\nu=(w,r)\in \ZZ^{n+1}$ be a toric Reeb vector and let $\varphi_t: M\to M$ be the flow of $-R_\nu$; this flow is 1-periodic and induces the action of $S^1=\RR/\ZZ$ on $M$. Denote by $M^T$ the fixed point set of $\varphi_T$ and let $B_T=M^T/S^1$. Then the inertia orbifold of $B$ is
\[\bigsqcup_{0<T\leq 1} B_T.\]
In particular, each $B_T$, $0<T<1$, is the disjoint union of twisted sectors and $B_0=B_1=B$. We denote by $F_T^\ast$ the contribution of the twisted sectors contained in $B_T$ to $H^\ast_{orb}(B)$, so that 
\begin{equation}\label{eq: baseorbifolddecomposition}H_{orb}^\ast(B; \QQ)=\bigoplus_{0<T\leq 1} F_T^\ast.
\end{equation}
Note in particular that $F_0=F_1=H^\ast(B; \QQ)$ is the singular cohomology of $B$. 

\begin{theorem}
Suppose we have the conditions of Proposition \ref{rgorenstein}. Then the contact homology of $M$ is completely determined by the decomposition \eqref{eq: baseorbifolddecomposition} and is given by
$$\textup{HC}_\ast(M, \xi)=\bigoplus_{k\geq 0}\bigoplus_{0<T\leq 1}F_T^{\ast-2rT+2-2rk}.$$
\label{chfromorbifold}
\end{theorem} 
\begin{proof}

In this proof we will use the language and the results from appendices \ref{appendix} and \ref{appendixB}. Proposition \ref{orbicohboxsum} describes the orbifold cohomology of the base $B$ in terms of the combinatorics of its fan $\Sigma$. On the other hand, Lemma \ref{lem: ehrhartstartriangulation} describes the Ehrhart series of $D$ in terms of its combinatorics. Our strategy will be to compare the two formulas using the correspondence between cones of $\Sigma$ and faces of $D$. Indeed, we have the following bijections:
\[\{\textup{face of codim}= k\textup{ of }\Delta\}\leftrightarrow\{\textup{cone of dim}=k\textup{ of }\Sigma\}\leftrightarrow \{\textup{face of dim}=k-1\textup{ of }D\}\,.\]
A face of $\Delta$ with normals $v_1, \ldots, v_k$ corresponds to the cone $\tau$ of $\Sigma$ with rays $v_i$ which in turn corresponds to the face $g$ of $D$ with vertices $u_1= A\nu_1,\ldots, u_k=A\nu_k$ (see Remark \ref{rem: deltaDcorrespondence}). We note that under this correspondence there is a bijection between $\bx(\tau)$ (as defined in \eqref{eq: boxtau}) and $\bx(g)$ (as defined in \eqref{eq: boxg}). If
\[\sum_{i=1}^k c_i v_i\in \bx(\tau)\subseteq \ZZ^n\]
then there is a unique $0\leq T<1$ such that
\begin{equation}\label{eq: Tdefinition}\sum_{i=1}^k c_i\nu_i+Te_{n+1}\in \ZZ^{n+1}\,.
\end{equation}
Applying the matrix $$\begin{bmatrix}
A\\ 
\begin{matrix}
-w-&
r
\end{matrix}
\end{bmatrix}\in GL_{n+1}(\ZZ)$$
from Remark \ref{rem: deltaDcorrespondence} we get
\[\sum_{i=1}^k c_i \mu_i+T(w,r)\in \bx(g).\]

Recall that the twisted sectors of $M$ are in bijection with
\[\bx(\Sigma)=\bigcup_{\tau\in \Sigma}\bx(\tau)=\bigcup_{g\subseteq D}\bx(g).\]

\begin{lemma}\label{lem: twistedsectorsbox}
The twisted sectors contributing to the summand $F_T^\ast$ of $H^\ast_{orb}(B)$ are the ones corresponding to elements of \[\bx_T(g)=\{\mu\in \bx(g): T(\mu)=T\}\] 
for some face $g$ of $D$. 
\end{lemma}
\begin{proof}
Consider the commutative diagram of exact sequences
\begin{center}
\begin{tikzcd}
& & S^1\arrow[d, "j"]\\
\widetilde K \arrow[r]\arrow[d]& \TT^d\arrow[r, "\widetilde \beta"]\arrow[d, equal] & \TT^{n+1}\arrow[d]\\
 K \arrow[r]\arrow[d, "\phi"]& \TT^d \arrow[r, "\beta"]& \TT^{n}\\
S^1&  & 
\end{tikzcd}
\end{center}
where 
\[\beta\left(\sum_{i=1}^d c_i e_i\right)=\sum_{i=1}^d c_i v_i\quad,\quad\widetilde \beta\left(\sum_{i=1}^d c_i e_i\right)=\sum_{i=1}^d c_i \nu_i\,,\]
and $K, \widetilde K$ are the kernels of $\beta$ and $\widetilde \beta$, respectively.  The map $j\colon S^1\to \TT^{n+1}$ sends $t\in \RR/\ZZ=S^1$ to $-t e_{n+1}\in \RR^{n+1}/\ZZ^{n+1}=\TT^{n+1}$ and the map $\phi$ is induced by the diagram: if $\kappa\in K$ then $\widetilde \beta(\kappa)=j(t)$ for a unique $t\in S^1$, so we take $\phi(\kappa)=t$. We let $Z=\mu_K^{-1}(b)\subseteq \CC^d$, then we have $M=Z/\widetilde K$ and $B=Z/K$ (see the Appendix \ref{appendix: toricorbifolds} and note that the symplectization of $M$ is by definition the symplectic reduction of $\CC^d$ by $\widetilde K$). The action of $S^1=K/\widetilde K$ on $M=Z/\widetilde K$ is precisely the action of $S^1$ on $M$ induced by the Reeb flow of $-R_{\nu}$.  

As explained in Appendix \ref{appendix}, the twisted sectors of $B$ correspond to elements $\kappa=\sum_{i=1}^d c_i e_i \in K$ such that $Z^\kappa\neq \emptyset$, which in turn correspond to an element
\[\sum_{i\in R(\tau)} c_i v_i\in \bx(\tau)\]
for the cone $\tau\in \Sigma$ generated by the rays $v_i$ for which $c_i\neq 0$. If we let $T$ be such that \eqref{eq: Tdefinition} holds, then
\[\widetilde\beta(\kappa)=\sum_{i=1}^d c_i\nu_i=-T e_{n+1}=j(T)\quad\textup{in }\TT^{n+1}\,,\]
so $\phi(\kappa)=T$. Hence the twisted sector corresponding to $\kappa$ is contained in $B_T=M^T/S^1$, which proves the Lemma.\qedhere
\end{proof}

We denote by $P_q(B_T)$ the Poincaré polynomial of $B_T$ with the orbifold shift, i.e. $P_q(B_T)= \sum_{j\in \QQ} \dim F_T^{2j} q^j$. By (the proof of) Proposition \ref{orbicohboxsum}, Lemma \ref{lem: twistedsectorsbox} and the correspondence between cones of $\Sigma$ and faces of $D$, we have:
\[P_q(B_T)=\sum_{g\subseteq D}\left(\sum_{f\supseteq g}q^{\dim f-\dim g}(1-q)^{n-\dim f-1}\right)\sum_{\mu\in \bx_T(g)}q^{\psi(\mu)}.\]
 By Lemma \ref{lem: ehrhartstartriangulation} it follows that
\[\frac{\sum_{j=0}^n \delta_j q^j}{1-q}=(1-q)^n\Ehr_D(q)=\sum_{0\leq T<1}\frac{q^{rT}}{1-q^{r}}P_q(B_T)\,. \]
We will now replace $q$ by $q^{-1}$ in the previous equality and multiply both sides by $q^n$. The equation then becomes
\begin{equation}\label{eq: proofprequantization}
\frac{q}{1-q}\sum_{j=0}^n \delta_{n-j}q^j=\sum_{0\leq T<1}\frac{q^{r(1-T)+n}}{1-q^{r}}P_{q^{-1}}(B_T).\end{equation}
Note that the left hand side is equal to
\[\frac{q}{1-q}\sum_{j=0}^n \delta_{n-j}q^j=\sum_{j\geq 1}cb_{2j-2}q^j\]
by Theorem \ref{main}. 
By orbifold Poincaré duality \cite[Proposition 3.3.1]{chen2004new}, we have
\[q^n P_{q^{-1}}(B_T)=P_q(B_{1-T})\,.\]
Thus the right hand side of \eqref{eq: proofprequantization} is equal to
\[\sum_{0\leq T<1}\frac{q^{r(1-T)}}{1-q^{r}}P_{q}(B_{1-T})=\sum_{0< T\leq 1}\frac{q^{rT}}{1-q^{r}}P_{q}(B_{T})\]
It follows from comparing the $q^{j+1}$ coefficient on both sides of \eqref{eq: proofprequantization} that
\[cb_{2j}=\sum_{0<T\leq 1}\sum_{k\geq 0}\dim F_T^{2j-2rT-2kT+2},\]
which proves Theorem \ref{chfromorbifold}.\qedhere
\end{proof}

As observed in the introduction, Theorem \ref{chfromorbifold} has two corollaries that were already known. The case $r=1$ is a consequence of Stapledon's result \cite[Theorem 4.3]{stapledon} together with Theorem~\ref{thm:invariants}, see Corollary \ref{cor:Stapledon}. When the base $B$ is smooth, and assuming that contact homology is a well-defined contact invariant, our result specializes to \cite[Proposition 9.1]{bourgeois2002morse}, see Corollary~\ref{cor:Bourgeois}.

\begin{remark}
Suppose that $r=1$ and $B$ is smooth; that is, suppose that $\Delta$ is a reflexive Delzant polytope. Then by Poincaré duality in $B$ we have
$$\delta_j=\dim H^{2j}(B; \QQ)=\dim H^{2(n-j)}(B; \QQ)=\delta_{n-j}.$$
This is a manifestation of Hibi's palindromic theorem, cf. Theorem \ref{hibbi}.

More generally, suppose $r>1$ but $B$ is still smooth, so that Corollary \ref{cor:Bourgeois} applies. For simplicity of notation write $d_j=\dim \textup{HC}_{2j}(M, \xi)$ and $b_j=\dim H_{2j}(B; \QQ)$. We have:

\begin{align*}
\sum_{j=0}^n \delta_{n-j} z^j&=(1-z)\left(\sum_{j\geq 0} d_j z^j\right)=(1-z)\left(\sum_{i=0}^n\sum_{k\geq 0} b_j z^{j+(r-1)+rk}\right)\\
&=z^{r-1}\frac{1-z}{1-z^r}\left(\sum_{i=0}^n b_j z^j\right).
\end{align*}
We used Theorem \ref{main} in the first equality and Corollary \ref{cor:Bourgeois} in the second. Hence $\delta_n=\ldots=\delta_{n-r+2}=0$. Moreover since $\sum_{i=0}^n b_j z^j$ and $\frac{1-z^r}{1-z}$ are both palindromic polynomials it follows that $\frac{1}{z^{r-1}}\sum_{j=0}^n \delta_{n-j} z^j$ is also palindromic, that is, $\delta_{j}=\delta_{n-r+1-j}$ for $j=0, 1, \ldots, n-r+1$. This can be seen from the extension of Hibi's palindromic theorem for Gorenstein polytopes in \cite[Theorem 4]{lee2015extension}.
\end{remark}

\subsection{Examples with a smooth base}

We begin by considering the only two examples of Delzant $r$-Gorenstein polytopes with $n=2$ and $r>1$. These two are the moment polytopes of $(\CC P^2, \omega_{FS})$  and $(S^2\times S^2, \omega_0\oplus \omega_0)$ where $\omega_0$ is the volume form on $S^2$ with total volume $1$. These two have $r=3$ and $r=2$, respectively. These reflexive polytopes are the moment polytopes of $(\CC P^2, 3\omega_{FS})$ and $(S^2\times S^2, 2\omega_0\oplus 2\omega_0)$.

One can construct the toric diagrams of the prequantizations of these four cases as explained above. All of these correspond to familiar contact manifolds according to \cite{abreu2018contact},
namely the sphere $S^5$, the unit cosphere bundle $S^\ast S^3\cong S^2\times S^3$, the Lens space $L^5_3(1,1,1)$ and the unit cosphere bundle $S^\ast \RR P^3\cong S^2\times \RR P^3$.

Now Corollary \ref{cor:Bourgeois} gives the contact homology of these four contact manifolds as sums of the homology of the base. We start with $S^5$ which is the prequantization of $\CC P^2$ with $r=3$. The Betti numbers of $\CC P^2$ are $\dim H_\ast(\CC P^2; \QQ)=1$ if $\ast=0,2,4$ and $0$ otherwise. So the contact Betti numbers of $S^5$ are given as follows:
$$\begin{matrix}
 &  &1  & 1 & 1 & & &\\ 
 &  &  &  &  &1 & 1&1 \\ 
 &&&&&&&&\ddots\\
\hline
 0& 0 & 1 & 1 & 1 & 1&1&1& \ldots
\end{matrix}$$
This should be interpreted as follows: the columns correspond to (even) degrees of homology or contact homology, the first lines are the copies of the homology of the base with the shifts given by Corollary \ref{cor:Bourgeois}, and the last line is the contact homology. So the final result is that $\dim \textup{HC}_{2j}(S^5)=1$ if $j\geq 2$ and $0$ if $j=0, 1$. We play the same game but now for the prequantization of $(\CC P^2, 3\omega_{FS})$ (that is, with $r=1$) to get the following contact Betti numbers for $L^5_3(1,1,1)$:
$$\begin{matrix}
1 &1  &1  &  &  & \\ 
 &1  &1  &1  &  & \\ 
 &  & 1 &1  &1  & \\ 
 &&&&&\ddots\\
\hline
 1& 2 & 3 & 3 & 3 & \ldots
\end{matrix}$$
For $S^\ast S^3$ as the prequantization of $(S^2\times S^2, \omega_0\oplus \omega_0)$ we get the following:
$$\begin{matrix}
 & 1 &2  & 1 &  & \\ 
 &  &  &1  & 2 & 1\\ 
 &&&&&&\ddots\\
\hline
 0& 1 & 2 & 3 & 3 &3& \ldots
\end{matrix}$$
And for $S^\ast \RR P^3$ as the prequantization of $(S^2\times S^2, 2\omega_0\oplus 2\omega_0)$ we find:
$$\begin{matrix}
1 & 2 &1  &  &  & \\ 
 &1  &2  &1  &  & \\ 
 &  & 1 &2  &1  & \\ 
 &&&&&\ddots\\
\hline
 1& 3 & 4 & 4 & 4 & \ldots
\end{matrix}$$

\subsection{Example with orbifold base and $r>1$} \label{topexample}
We consider now an example with $r>1$. We begin again with the lens space $M=L_3^5(1,1,1)$ with toric diagram $D=\textup{conv}(v_1, v_2, v_3)$ with $v_1=(0,0)$, $v_2=(1,0)$ and $v_3=(2,3)$. Take the Reeb vector $\nu=(1,1,2)$, corresponding to the point in the interior of $D$ with coordinates $(1/2, 1/2)$, and let $B$ be the orbifold obtained by quotienting $M$ by the Reeb action. In this situation $r=2$. To find $\Delta$ we perform a change of coordinates in $GL(3, \ZZ)$ sending $\nu$ to $e_3$, for instance:
$$\begin{bmatrix}
-2 &0 &1 \\ 
-1 & 1 & 0\\ 
1 &0  &0 
\end{bmatrix}\begin{bmatrix}
1 &0  &  1&2 \\ 
1 & 0 & 0 & 3\\ 
2 &  1&1  & 1
\end{bmatrix}=\begin{bmatrix}
 0& 1 & -1 & -3\\ 
 0& 0 & -1 & 1\\ 
 1&0  & 1 & 2
\end{bmatrix}\,.$$
Thus 
$$\Delta=\{(x,y)\in \RR^2: x\geq 0, -x-y+1\geq 0, -3x+y+2\geq 0\}.$$
The only point with non-trivial isotropy is the point that maps to the vertex $(3/4, 1/4)\in \Delta$ that corresponds to the edge normal to $\nu_2, \nu_3$ in $C$. Since
$$\begin{vmatrix}1&1&2\\
1&0&3\\
2&1&1\end{vmatrix}=4$$
the structure group at this point is $\ZZ/4\ZZ$ and the only $T$ for which $B_T\neq \emptyset$ are $T=1/4, 2/4, 3/4, 1$. It can be seen that $B$ is the weighted projective space $\CC P^2(4,1,1)$. The non-twisted sector is $B_1=\CC P^2(4,1,1)$ and the twisted sectors are all just a point $B_{1/4}=B_{2/4}=B_{3/4}=\{\ast\}$. These three twisted sectors $B_{k/4}$ correspond to
$$\frac{k}{4}\nu_2+\frac{k}{4}\nu_3+\frac{k}{4}\nu=(k,k,k)\in \bx_{k/4}(\Sigma)\subseteq \ZZ^3 \textup{ for }k=1,2,3.$$
Thus the degree shift associated to $B_{k/4}$ is $2\big(\frac{k}{4}+\frac{k}{4}\big)=k$. Recalling that a weighted projective space has the same singular cohomology as the smooth projective space,
$$F_1=\begin{cases}\QQ &\textup{ if }\ast=0,2,4\\
0 &\textup{ otherwise}
\end{cases} \textup{ and }\, F_{k/4}=\begin{cases}\QQ &\textup{ if }\ast=k\\
0 &\textup{ otherwise}
\end{cases} \textup{ for }k=1,2,3$$
 so the Chen-Ruan cohomology of $\CC P^2(4,1,1)$ is obtained as the sum of these:
\begin{align*}H^\ast_{orb}(\CC P^2(4,1,1); \QQ)=\begin{cases}\QQ &\textup{ if }\ast=0,1,3,4\\
\QQ\oplus \QQ &\textup{ if }\ast=2\\
0 &\textup{ otherwise}
\end{cases}.\end{align*}

Now Theorem \ref{chfromorbifold} computes $\textup{HC}_\ast(L^5_3(1,1,1))$ as a sum with contributions from the different sectors:

$$\begin{matrix}
\,\,\,\;F_1\,\vline&   & 1 & 1 & 2 & 1 & 2 & 1 \\
F_{1/4}\,\vline& 1 &   & 1 &   & 1 &   & 1 \\
F_{2/4}\,\vline&   & 1 &   & 1 &   & 1 &   \\
F_{3/4}\,\vline&  &   & 1 &   & 1 &   & 1 \\ \hline
&1  & 2  & 3 & 3 & 3 & 3 & 3
\end{matrix}$$

Here the first four lines are the ranks contributed by the twisted sector corresponding to $T$ according to Theorem~\ref{chfromorbifold} and the last line is the total rank of the contact homology. For example the first line means
$$F_1^{\ast-2}\oplus F_1^{\ast-6}\oplus F_1^{\ast-10}\oplus \ldots=\begin{cases}\QQ\oplus \QQ &\textup{ if }\ast=4k+2, k>0\\
\QQ &\textup{ if }\ast=2 \textup{ or }\ast=4k, k>0\\
0 &\textup{ otherwise.}
\end{cases}$$
Note that the final result obtained coincides with the previous computation of $\textup{HC}_\ast (L^5_3(1,1,1))$ as one would expect.

\appendix
\section{Chen-Ruan Cohomology of Toric Orbifolds}
\label{appendix}

In this appendix we describe how to compute the Chen-Ruan cohomology of a toric orbifold given its moment polytope. Before that we explain briefly the construction of orbifold cohomology. We refer to \cite{adem2007orbifolds} for more details.

We think of orbifolds as (equivalence classes of) Lie groupoids $\G$. Recall that a Lie groupoid essentially consists of an object space $G_0$, a morphism space $G_1$ and source and target arrows $s,t: G_1\to G_0$ obeying certain conditions. We say that $\G$ is an étale groupoid if $s,t$ are local isomorphisms. The orbifold associated to $\G$ is to be interpreted as $G_0/\!\!\sim$ where $\sim$ is the equivalence relation defined by $s(g)\sim t(g)$ for $g\in G_1$. 

Two groupoids give the same orbifold if they are Morita equivalent. If we establish an analogy between groupoids as the underlying structure of an orbifold and atlases as the underlying structure of a manifold, then Morita equivalence is analogous to the existence of a common refinement of the atlases. Note that a non-étale groupoid can be Morita equivalent to an étale one, and we still say that such groupoid represents an orbifold. So orbifolds are Morita equivalence classes of orbifolds containing some étale groupoid. We denote by $|\G|$ the underlying topological space of the orbifold.

\subsection{Quotient orbifolds}

Let $X$ be a manifold and $G$ be a compact Lie group acting almost freely on $X$. Then $X/G$ is an orbifold with groupoid structure $\G=G\ltimes X$ where its object space is $X$, its morphism space is $G\times X$ and $s(g,x)=x, \, t(g,x)=gx$. We assume for simplicity that $G$ is abelian, which is the case we are interested in. Then the inertia groupoid of $\G$ is 
$$\Lambda \G=\bigsqcup_{g\in G}\G^g$$
where $\G^g=G\ltimes X^g$ (note that for $G$ abelian each conjugacy class has a unique element). The connected components $\G^g$ with $g\neq 0$ of the inertia groupoid are called the twisted sectors. Now the Chen-Ruan cohomology of $\G$ is defined as the (singular) cohomology of $\Lambda \G$ but with appropriate degree shifting in each twisted sector. More precisely
$$H^d_{orb}(\G)=\bigoplus_{g\in G}H^{d-2\iota_g}(|\G^g|)$$
for some degree shifting numbers $\iota_g\in \QQ$ (also called age grading in the literature). Note that $|\G^g|$ is the quotient topological space $X^g/G$; in particular if $X^g$ is empty then there is no contribution of $g$. We also note that the use of a non-étale groupoid does not affect in any way this construction except in the definition of the degree shifting numbers, since a Morita equivalence of groupoids induces a Morita equivalence of inertia groupoids and hence homeomorphisms between the underlying topological space of their components.

To define the degree-shifting numbers in this setup we have to get an orbifold local chart from this non-étale groupoid. Let $g\in G$ and $x\in X^g$. Take a slice $S$ at $x$. Then $(S, G_x, \pi: S\to S/G_x)$ is an orbifold chart around $x$. We assume that $S$ has an (almost) complex structure; note that this does not have to be inherited from $X$. Then $g$ acts infinitesimally on $T_x S\cong \CC^n$ and has eigenvalues $e^{2\pi i \lambda_1}, \ldots, e^{2\pi i \lambda_n}$ with $0\leq \lambda_j<1$. Now the degree shifting number is defined to be
$$\iota_g=\sum_{j=1}^n \lambda_j.$$
This is independent of the point $x\in X^g$ chosen.

\subsection{Toric orbifolds}
\label{appendix: toricorbifolds}

We now explain how to compute the orbifold cohomology in our cases of interest of toric orbifolds. We can consider for instance a compact toric symplectic orbifold $B$ with labeled moment polytope $\Delta$ as in section \ref{baseorbifold} or take $B=X_\Sigma$ and $\Delta=P$ as in section \ref{resolutions}. 

The construction of $B$ from 
$$\Delta=\{x\in \RR^{n}: \langle x, v_i \rangle+b_i\leq 0, \textup{ for } i=1, \ldots, d\}$$
(where, as in section \ref{baseorbifold}, the $v_i$ are the primitive vector times the corresponding label) is the following, as described in \cite{lerman1997hamiltonian}: consider the torus $\TT^d=\RR^d/\ZZ^d$ and its action on $\CC^d$ by
$$(c_1, \ldots, c_n)\cdot (z_1, \ldots, z_n)=\left(e^{2\pi i c_1}z_1, \ldots, e^{2\pi i c_d}z_d\right).$$
Consider the map $\beta: \TT^d\to \TT^n$ defined by $\beta(e_j)=v_j$ and let 
$$K=\ker \beta=\left\{\sum_{i=1}^d c_i e_i: \sum_{i=1}^d c_iv_i\in \ZZ^n\right\}\subseteq \TT^d.$$
Then $B$ is defined to be the symplectic reduction of $\CC^d$ by $K$. Hence $B$ is the orbifold with groupoid structure $K\ltimes Z$ where $Z=\mu_K^{-1}(b)$ is the pre-image of an appropriate point $b\in \mathfrak k^\ast$ in the dual of the Lie algebra of $K$. 

Now take $\kappa=\sum_{i=1}^d c_ie_i\in K$ and assume that $Z^\kappa\neq \emptyset$. Without loss of generality we assume that $c_1, \ldots, c_k\neq 0$ and $c_{k+1}=\ldots=c_n=0$. Then $z\in Z^\kappa$ if and only if $z_1=\ldots=z_k=0$, that is, if and only if $\mu_K([z])$ is in the face $F$ of $\Delta$ normal to $v_1, \ldots, v_k$ (in particular such face must be non-empty). The twisted sector $Z^\kappa/K$ is equivalent (up to a non-effective quotient) to the orbifold with moment polytope $F$, and hence its singular cohomology can be read from the combinatorics of $F$ using a Morse function as explained in \cite[section 3.3]{cannas2001xi}.

It remains to compute $\iota_\kappa$. Let $S\subseteq Z$ be a slice at $z$; concretely, if we assume the facets normal to $v_1, \ldots, v_n$ intersect at a vertex, one may take
$$S=Z\cap (\CC^n\times\RR^{d-n}).$$ We have $T_z Z=T_z S\oplus T_z(K\cdot z)$. Since $T_z Z$ is $\omega$-orthogonal to $T_z(K\cdot z)$ (see for instance \cite[section 23.3]{da2001lectures}) it follows that $T_zZ$ is orthogonal to $i T_z(K\cdot z)$. Hence
\begin{align*}T_z \CC^d&=T_z S\oplus T_z(K\cdot z)\oplus iT_z(K\cdot z)\\
&=T_z S\oplus (T_z(K\cdot z)\otimes \CC)\end{align*}
Now the infinitesimal action of $\kappa$ on $T_z\CC^d\cong \CC^d$ is given by the matrix $$\textup{diag}(e^{2\pi i c_1}, \ldots, e^{2\pi i c_k}, 1 \ldots, 1).$$
Moreover $\kappa$ acts trivially on $K\cdot z$ since $\kappa\cdot\lambda\cdot z=\lambda\cdot \kappa\cdot z=\lambda\cdot z$ for any $\lambda\in K$. Considering the splitting of $T_z\CC^d$ described above it follows that 
$$\iota_\kappa=\sum_{j=1}^k \{c_j\}.$$

\subsection{Formula in terms of Box sums}

The previous description of the orbifold cohomology of a toric orbifold can be packed efficiently in terms of a sum over the Box of the fan corresponding to $\Delta$, as in \cite[section 4]{stapledon}. We let $\Sigma$ be the (simplicial) fan associated to $\Delta$. The cones of $\Sigma$ are in bijection with faces of $\Delta$; more precisely, if $\tau$ is the cone with rays $v_1, \ldots, v_k$ then the corresponding face of $\Delta$ is
$$F=\{v\in P: \langle v, v_i\rangle +b_i=0\}$$
for some $a_1, \ldots, a_k$.

If $\tau\in \Sigma$ is a cone we define its $h$-polynomial as
$$h_\tau(q)=\sum_{\sigma\supseteq \tau} q^{\dim \sigma-\dim \tau}(1-q)^{\codim \sigma}.$$
Note that the $h$-polynomial of $\tau$ is the same as the $h$-polynomial of the polytope determined by the corresponding face $F$:
$$h_\tau(q)=h_F(q)\equiv \sum_{G\subseteq F}q^{\dim(F)-\dim(G)}(1-q)^{\dim(G)}.$$

 Given a cone $\tau\in \Sigma$ with rays $v_1, \ldots, v_k$ we define the box of $\tau$ to be the set
\begin{equation}
\label{eq: boxtau}
\bx(\tau)=\left\{\sum_{j=1}^k c_j v_j: 0<c_j<1\right\}\cap \ZZ^{n+1}.
\end{equation}
Given $\kappa\in\bx(\tau)$ we let
$$\psi(\kappa)=\sum_{j=1}^k c_j \in \QQ.$$

\begin{proposition}
\label{orbicohboxsum}
The orbifold Poincaré polynomial of $X_\Sigma$ is given by
$$\sum_{j\in \QQ}\dim H^{2j}_{orb}(X_\Sigma)q^j=\sum_{\tau\in \Sigma}h_\tau(q)\sum_{\kappa\in \bx(\tau)}q^{\psi(\kappa)}.$$
\end{proposition}
\begin{proof}
The proof is essentially contained in \cite[section 4]{stapledon}, but is also easy after our description of the orbifold cohomology. We showed before that twisted sectors are in bijection with $\bx(\Sigma)=\bigsqcup_{\tau\in \Sigma} \bx(\tau)$ and the twisted sector corresponding to $\kappa\in \bx(\tau)$ is (up to a trivial quotient) $X_{\Sigma_\tau}$. By \cite[Lemma 4.1]{stapledon}, $h_\tau(q)$ is the Poincaré polynomial (the usual one, not  orbifold) of $X_{\Sigma_\tau}$ and as shown before $\psi(\kappa)$ is the shift associated to the twisted sector.
\end{proof}

\section{Generalizing Stapledon's theorem}
\label{appendixB}

In this appendix we explain how to adapt the proof of \cite[Theorem 4.6]{stapledon} to allow $D$ to be a rational polytope. This corresponds to the case where the toric contact manifold $(M, \xi)$ is $\QQ$-Gorenstein, i.e. $c_1(\xi)$ is torsion, and $c_1(X_{\Sigma})$ is torsion, where $X_{\Sigma}$ is the resolution of the symplectic cone of $M$ associated to a triangulation of $D$. 

Let $D\subseteq \RR^n$ be a polytope with rational vertices $v_1, \ldots, v_d\in \mathbb{Q}^n$. We consider also a rational triangulation $\T$ of $D$. Assume that $m\in \mathbb Z^+$ is an integer such that $mD$ and $m\T$ are an integral polytope and an integral triangulation, respectively. We let $\Sigma$ be the (stacky) fan over $\T$; $\Sigma$ has rays
$$\{\nu=(mv,m)\in \ZZ^{n+1}: v\in \T_0\}.$$

\begin{theorem}
Let $D$ be a rational toric diagram and $\T$ a rational triangulation of $D$. Assume that $m$ is such that $mD$ has integral vertices and $m\T$ is an integral triangulation. Let $\Sigma$ the fan over $\T$ and $X_\Sigma$ the toric variety with fan $\Sigma$. Then 
$$\dim H^{2j}_{\textup{orb}}(X_\Sigma; \QQ)=\delta_{mj}$$
for $j\in \frac{1}{m}\ZZ$, and the remaining orbifold cohomology groups are trivial. 
\label{stapledongeneral}
\end{theorem}
\begin{proof}
We will prove the result by giving a formula for the Ehrhart series of $D$ in terms of a Box sum and then comparing with Proposition \ref{orbicohboxsum}.

For each vertex $v_j\in \frac{1}{m}\ZZ^n$ of $D$ let $\nu_j=(mv_j, m)$. Integral points in $tD\cap \ZZ^n$, for $t\in \mathbb Z_{\geq 0}$, are in bijection with integral points in the cone 
$$C=\{(tv, t): v\in D, t\geq 0\}=\mathbb R_{\geq 0}\langle \nu_1, \ldots, \nu_d\rangle$$
with last coordinate $t$.
If $\nu=(tv,t)\in C$ is an integral vector let $\theta\in \T$ be the face containing $v$ in its relative interior. Then there is a unique way to write $\nu$ as
$$\nu=\kappa+\sum_{\nu_j\in R(\theta)\setminus R(\eta)} \nu_j+\sum_{\nu_j\in R(\theta)}m_j \nu_j$$
where $\eta\subseteq \theta$, $\kappa\in \bx(\eta)$ and $m_j\in \mathbb{Z}_{\geq 0}$ for $\nu_j\in R(\theta)$. Above, we wrote $R(\theta)$ for the set of rays $\nu_j=(mv_j,m)$ for $v_j\in \theta_0\subseteq \T_0$ a vertex of $\theta$.
The last coordinate of such $\nu$ is
$$m\left(\psi(\kappa)+\dim\theta-\dim \eta+\sum_{\nu_j\in R(\theta)}m_j\right).$$
Hence the Ehrhart series of $D$ is
$$\sum_{\theta\in \Sigma} \sum_{\eta\subseteq \theta} \frac{q^{m(\dim \theta-\dim \eta)}}{(1-q^m)^{\dim(\theta)+1}}\sum_{\mu\in \bx(\eta)}q^{m\psi(\mu)}=\frac{1}{(1-q^m)^{n+1}}\sum_{\eta\in \Sigma}h_\eta(q^m)\sum_{\kappa\in \bx(\eta)}q^{m\psi(\kappa)}.$$
Comparing with Proposition \ref{orbicohboxsum} it follows that
\[\sum_{j\in \frac{1}{m}\ZZ}\delta_{mj}q^{mj}=(1-q^m)^{n+1}\Ehr_D(q)=\sum_{\eta\in \Sigma}h_\eta(q^m)\sum_{\kappa\in \bx(\eta)}q^{m\psi(\kappa)}=\sum_{j\in \frac{1}{m}\ZZ} \dim H^{2j}_{\textup{orb}}(X_\Sigma)q^{mj}\,\]
and we are done.\qedhere
\end{proof}

\subsection{Ehrhart series from star subdivision}

We will now prove a Lemma that provides a formula for the Ehrhart series of an integral polytope $D$ when we are given a rational point $w/r$ in the interior of $D$. This will be used in the proof of Theorem \ref{thm:quotients}. The statement and proof are slight variations to those of Theorem \ref{stapledongeneral}: we will consider the star subdivision centered at $w/r$ and obtain a formula by essentially the same argument. 

To state the formula we introduce some notation. Given a face $g\subseteq D$ with vertices $v_1, \ldots, v_k$, we denote by $\bx(g)$ the set
\begin{equation}\label{eq: boxg}\bx(g)=\left\{\sum_{j=1}^k c_j \nu_j+T (w,r)\colon 0<c_j<1\,,\, 0\leq T<1\right\}\cap \ZZ^{n+1}\end{equation}
where $\nu_j=(v_j, 1)$. Given $\kappa\in \bx(g)$ written as $\kappa=\sum_{j=1}^k c_j \nu_j+T (w,r)$, we denote
\[\psi(\kappa)=\sum_{j=1}^k c_j\quad,\quad T(\kappa)=T.\]

\begin{lemma}\label{lem: ehrhartstartriangulation}
Let $D$ be a toric diagram with integral vertices and let $w\in \ZZ^n$, $r\in \ZZ^+$ be such that $w/r\in \inte D.$ Then the Ehrhart series of $D$ is given by

\[\Ehr_D(q)=\frac{1}{1-q^r}\sum_{g\subseteq D}\left(\sum_{f\supseteq g}q^{\dim f-\dim g}(1-q)^{-\dim f-1}\right)\sum_{\kappa\in \bx(g)}q^{\psi(\kappa)+rT(\kappa)}\]
where the first sums run over faces $g\subseteq f$ of $D$.
\end{lemma}
\begin{proof}
For each face $f$ of $D$ with vertices $v_1, \ldots, v_\ell$ we let $D_f\subseteq D$ be the convex hull of the relative interior of $f$ together with $w/r$, that is,
\[D_f=\left\{\sum_{i=1}^\ell a_j v_j+a \frac{w}{r}\,: a_j>0\,,\, a\geq 0\,, \,\sum_{j=1}^\ell a_j+a=1\right\}.\]
The polytope $D$ is the disjoint union 
\[D=\bigsqcup_{f\subseteq D} D_f\,,\]
hence
\[\Ehr_D(q)=\sum_{f\subseteq D}\Ehr_{D_f}(q).\]
Integral points in $t D_f\cap \ZZ^n$ are in bijection with integral points in 
\[C_f=\{(t v, t): v\in D_f\,,t\geq 0\}\]
with last coordinate $t$. Such points can be written uniquely as
\[\nu=\kappa+\sum_{v_j\in V(f)\setminus V(g)} \nu_j+\sum_{v_j\in V(f)} m_j \nu_j+m(w,r)\]
for some face $g\subseteq f$, $\kappa\in \bx(g)$ and $m_j, m\in \ZZ_{\geq 0}$; here we denote by $V(g)$ the set of vertices of the face $g$. The last coordinate of such point is
\[\psi(\kappa)+rT(\kappa)+\dim(f)-\dim(g)+\sum_{v_j\in V(f)} m_j+rm.\]
It follows that
\[\Ehr_{D_f}(q)=\sum_{g\subseteq f}q^{\dim(f)-\dim(g)}(1-q)^{-\dim(f)-1}(1-q^r)^{-1}\sum_{\kappa\in \bx(g)} q^{\psi(\kappa)+rT(\kappa)}.\]

After summing over $f$ we obtain the statement of the Lemma.\qedhere
\end{proof}

\section*{Acknowledgement}

 	We are grateful to the referees for several suggestions that improved the exposition,
in particular to one of them for making us give an entirely self-contained and purely
combinatorial proof of Theorem \ref{thm:quotients}.

\bibliographystyle{plain}
\bibliography{bibliografiaEhrhart}

\end{document}